\documentclass[a4paper,12pt,reqno]{amsart}
\usepackage[T1]{fontenc}
\usepackage{comment}
\usepackage{amsmath,amssymb,latexsym,amsfonts,amsthm,comment, xspace, enumerate} 
\usepackage[english]{babel}
\usepackage{algorithmic, algorithm}
\usepackage[normalem]{ulem} 
\usepackage{amsbsy}
\usepackage[latin1]{inputenc}
\usepackage{graphicx}
\usepackage{a4wide}
\usepackage{mathtools}
\usepackage{tikz}

\usepackage{color}

\DeclareMathAlphabet\mathbfcal{OMS}{cmsy}{b}{n}
\input amssym.def
\input amssym

\usetikzlibrary{matrix,arrows}
\allowdisplaybreaks
 \newtheorem{thm}{Theorem}[section]
 \newtheorem{defn}[thm]{Definition}
 \newtheorem{lem}[thm]{Lemma}
 \newtheorem{prop}[thm]{Proposition}
 \newtheorem{cor}[thm]{Corollary}
 \newtheorem{rem}[thm]{Remark}

 \newcommand{\bthm}{\begin{thm}}
 \newcommand{\ethm}{\end{thm}}
 \newcommand{\bd}{\begin{defin}}
 \newcommand{\ed}{\end{defin}}
 \newcommand{\blem}{\begin{lem}}
 \newcommand{\elem}{\end{lem}}
 \newcommand{\bcor}{\begin{cor}}
 \newcommand{\ecor}{\end{cor}}
 \newcommand{\bprop}{\begin{prop}}
 \newcommand{\eprop}{\end{prop}}
 \newcommand{\brem}{\begin{rem} \rm}
 \newcommand{\erem}{\end{rem}}
 \newcommand{\bex}{\begin{ex} \rm}
 \newcommand{\eex}{\end{ex}}

 \newcommand{\beq}{\begin{equation}}
 \newcommand{\eeq}{\end{equation} }

 \newcommand{\bea}{\begin{eqnarray}}
 \newcommand{\eea}{\end{eqnarray}}

 \newcommand{\beas}{\begin{eqnarray*}}
 \newcommand{\eeas}{\end{eqnarray*}}

 \newcommand{\beqs}{\begin{equation*}}
 \newcommand{\eeqs}{\end{equation*}}

 \newcommand{\bi}{\begin{itemize}}
 \newcommand{\ei}{\end{itemize}}

 \newcommand{\ben}{\begin{enumerate}}
 \newcommand{\een}{\end{enumerate}}

 \newcommand{\ba}{\begin{array}}
 \newcommand{\ea}{\end{array}}

 \newcommand{\ds}{\displaystyle}

 \newcommand{\N}{\mathbb N}

\newcommand{\NN}{\mathbb N}
\newcommand{\CC}{\mathbb C}
\newcommand{\RR}{\mathbb R}
\newcommand{\ZZ}{\mathbb Z}

\newcommand{\SSS}{\mathcal S}

\definecolor{darkred}{rgb}{0.8,0,0}
\def\col{\color{darkred}}

\begin{document}

\title{Ultradifferentiable functions
via the Laguerre operator}

\author{S. Jak\v si\'c}
           \address{Faculty of Forestry, University of Belgrade, Kneza Vi\v seslava 1, Belgrade, Serbia}
 \email{smiljana.jaksic@sfb.bg.ac.rs}

\author{S. Pilipovi\'c }
          \address{Department of Mathematics, Faculty of Sciences, University of Novi Sad, Trg Dositeja Obradovi\'ca 4, Novi Sad, Serbia}
          \email{stevan.pilipovic@dmi.uns.ac.rs}

   \author{N. Teofanov}
  \address{Department of Mathematics, Faculty of Sciences, University of Novi Sad, Trg Dositeja Obradovi\'ca 4, Novi Sad, Serbia}
  \email{tnenad@dmi.uns.ac.rs@dmi.uns.ac.rs}

\author[{\DJ}. Vu\v{c}kovi\' c]{\DJ or\dj e Vu\v{c}kovi\' c}
\address{\DJ or\dj e Vu\v ckovi\'c, Technical Faculty ``Mihajlo Pupin'', \DJ ure \DJ akovi\'{c}a bb, 23000 Zrenjanin, Serbia}
\email{djordjeplusja@gmail.com}

\begin{abstract}
We define and  characterize ultradifferentiable functions
and their corresponding ultradistributions on $\RR^d_+$ using iterates of the Laguerre operator.
The characterization is
based on decay or growth conditions of the coefficients in their Laguerre series expansion.
We apply our results to establish an isomorphism between
subspaces of Pilipovi\'c spaces on $\RR^d$, and the spaces of ultradifferentiable functions
on  $\RR^d_+$.
\end{abstract}



\maketitle


\section{Introduction} \label{sec:1}

Recently, Toft, Bhimani and Manna
demonstrated that (global) Pilipovi\'c spaces on $\RR^d$, $\mathcal H_{\alpha}(\RR ^d)$ and $\mathcal H_{0,\alpha}(\RR ^d)$, $\alpha > 0$,
and their dual spaces of distributions, appear as a convenient framework
in the study of the well-posednedness of certain types of Schr\"odinger equations,
which are ill-posed in Gelfand Shilov spaces of test functions and in their dual spaces of ultradistributions, see \cite{TBM}.
In this paper we address the following question:

Q: What are natural counterparts of global Pilipovi\'c spaces when considering the positive orthants
$\RR^d _+$ instead of $\RR^d$?

We recall that   $\mathcal H_{\alpha}(\RR ^d)$ and $\mathcal H_{0,\alpha}(\RR ^d)$, $\alpha > 0$,
are defined in terms of the Hermite operator, and introduce Pilipovi\'c spaces ($P-$spaces)  on positive orthants,  $\mathbf{G}^{\alpha}_{\alpha}(\RR^d_+)$ and  $\mathbf{g}^{\alpha}_{\alpha}(\RR^d_+)$,
$\alpha > 0$, of Roumieu and Beurling type by utilizing the iterates of the Laguerre operator.

We show that elements of  $\mathbf{G}^{\alpha}_{\alpha}(\RR^d_+)$ and  $\mathbf{g}^{\alpha}_{\alpha}(\RR^d_+)$, $\alpha > 0$, can be expressed through Laguerre series expansion, and establish a topological isomorphism between these spaces and sequence  spaces
$\ell_{\alpha}(\NN_0 ^d)$ and $\ell_{0,\alpha}(\NN_0 ^d)$ considered in \cite{Toft1}.
As a consequence of the Laguerre series expansion, we deduce that $\mathbf{G}^{\alpha}_{\alpha}(\RR^d_+)$ and  $\mathbf{g}^{\alpha}_{\alpha}(\RR^d_+)$ coincide with
$G-$type spaces  $G_\alpha^\alpha(\RR^d_+)$ and $g_\alpha^\alpha(\RR^d_+)$, respectively, when these spaces are nontrivial.  We refer to \cite{D4, SSB, SSBb} for $G-$type spaces.
As an independent side result, we prove that $g_1 ^1 (\RR^d_+) = \{ 0\}$.
While this  seems to be a known fact, we could not find an explicit proof in
the existing literature. Then we identify $(\mathbf{g}_1 ^1 (\RR^d_+))'$
as the smallest dual space of a $P-$space on $\RR^d_+$ that is not a (non-trivial)
$G-$type space.

Furthermore, to answer the question Q, we establish  a natural isomorphism between Pilipovi\'c spaces on $\RR^d$ consisting of even functions, and
$P-$spaces on $\RR^d _+$, extending main results from \cite{SSSB} where  an isomoprhism between subspaces of Gelfand-Shilov spaces and $G$-type spaces is shown.  Finally,  we conclude
the paper by exploring  the properties of $P-$flat
spaces on $\RR^d _+$, which are associated with  Pilipovi\'c flat
spaces on $\RR^d $ considered in e.g. \cite{Toft1}.

In addition, by using duality arguments we transfer our main results to the corresponding spaces of ultradistirbutions $(\mathbf{G}^{\alpha}_{\alpha})'(\RR^d_+)$ and
$(\mathbf{g}^{\alpha}_{\alpha})'(\RR^d_+)$.

\par

We may briefly summarize a background of our research as follows.

Spaces of smooth functions on $\RR^d_{+}$, which might serve as counterparts for the well-known spaces of Schwartz functions  $\mathcal S (\mathbb R^d) $, or Gelfand-Shilov functions $\Sigma_{\alpha}^{\alpha}(\mathbb R^d)$, $\alpha > 1/2$, and  $\SSS^{\alpha}_{\alpha}(\RR^d)$, $\alpha \geq 1/2$ (of Beurling and Roumieu  type, respectively)
have been studied for decades. In the 1990s, Duran studied Schwartz functions and tempered distributions on $\RR_{+}$, as well as Gelfand-Shilov type spaces and their distribution spaces on $\RR_{+}$ in \cite{D0, DLP, D3, D4, D1}. Comparing to the standard toolbox used to study
the spaces in the global (i.e.  $\RR^d$) setting,
different techniques are employed
when related spaces on orthants are considered. This includes, for example, Laguerre (eigen)expansions,  and  the Hankel-Clifford transform.
Duran's results were extended in \cite{SSB, SSBb, SSSB}, were the test spaces of ultradifferentiable functions of Beurling and Roumieu  type, denoted by  $g_\alpha^\alpha(\RR^d_+)$ and
$G_\alpha^\alpha(\RR^d_+)$, $\alpha\geq 1$, respectively, are introduced and characterized via decay properties of  coefficients in Laguerre series expansions of their elements.
The spaces  $g_\alpha^\alpha(\RR^d_+)$ and
$G_\alpha^\alpha(\RR^d_+)$ are non-trivial when $\alpha>1$ and $\alpha\geq 1$ respectively, and then we have dense and continuous inclusions:
$ g_\alpha^\alpha(\RR^d_+)\hookrightarrow G_\alpha^\alpha(\RR^d_+)$, $ \alpha >1$.

\par

In the global setting, the spaces of Gelfand - Shilov type
$\Sigma_{\alpha}^{\alpha}(\mathbb R^d)$, $ \alpha > 1/2$, and $\mathcal S_{\alpha}^{\alpha}(\mathbb R^d)$, $ \alpha \geq 1/2$, are characterized by the decay of their coefficients in Hermite (eigen)expansions, cf. \cite{lan, LCP, SP1988}. More generally, these spaces are characterized by the decay properties of the  coefficients
in expansions involving eigenfunctions of certain elliptic and positive Shubin-type operators $P(x,D)$, see \cite{GPR0,VV}.

Pilipovi\'{c} spaces $\mathcal H_{0,\alpha} (\mathbb R^d)$, $ \alpha > 0$, were initially introduced  in \cite{SP1988} by considering the iterates of the Hermite operator
(linear Harmonic oscillator). These spaces are nontrivial for any $ \alpha >0$, and for $\alpha \geq \frac{1}{2}$ they coincide with Gelfand-Shilov spaces $\Sigma_{\alpha}^{\alpha}(\mathbb R^d)$. Similar relations hold for the Roumieu case (with $\alpha > \frac{1}{2}$) as well, cf. \cite{Toft1}.
In this paper, we take a similar approach (see also \cite{SP,Zemanian}), and introduce test function spaces on  $\RR^d_+$ by using a suitable self-adjoint operator. Due to their close relation to Pilipovi\'{c} spaces on $\mathbb R^d$ we call them Pilipovi\'{c} spaces  on positive orthants, or $P-$spaces for short.

\par

The paper is organized as follows. In Section \ref{sec:2} we set the stage for our results by
collecting main notions which will be used in the sequel. We also prove some properties of sequence spaces (see subsection \ref{subsec:topology}), and the triviality of $g_1 ^1 (\RR^d_+) $, Theorem
\ref{triviality}.

$P-$spaces on
$\RR^d_+$, together with their representation in terms of Laguerre series expansions are introduced in Section \ref{sec3}. Theorems \ref{glavna} and \ref{glavnabeurling} extend existing results
\cite[Theorem 3.6]{D1}, \cite[Theorem 5.7]{SSB}, and \cite[Theorem 3.7]{SSBb}.
In Section \ref{sec4} we show that $L^2$-norm in the definition of  $P-$spaces can be replaced by any other $L^p$-norm, $ p\in [1,\infty]$,  Proposition \ref{equiv-norm1}. We establish in Section \ref{sec5} an isomorphism between subspaces of global Pilipovi\'c spaces and
$P-$spaces on $\RR^d_+$, Theorems \ref{even-spaces-1} and \ref{even-spaces-2}. These results  extend Propositions 3.1 and 3.2 from \cite{SSSB}.
Finally, in Section \ref{sec6} we introduce and study flat Pilipovi\'c spaces on orthants.



\section{Preliminaries} \label{sec:2}

In this section we collect some background  notions and properties which will be used in the sequel,
and start with a review of some basic notation.

\subsection{Notation}

We denote by $\NN $, $\ZZ $, $\RR $ and $\CC$ the sets of positive
integers, integers, real and complex numbers, respectively. Furthermore,
$\NN_0=\NN \cup\{0\}$, $\RR_+=(0,\infty)$,  $\overline{\RR_+}=[0,\infty)$.
The upper index $d \in \NN$ will denote the dimension,
$\RR^d_+=(0,\infty)^d$, etc.

We use the standard multi-index notation: for $\alpha\in\NN_0^d$ (or $ \alpha \in \overline{\RR^d_+}$)
and $x\in\RR^d$ (or $x\in\overline{\RR^d_+}$), $x^\alpha=x_1^{\alpha_1}...x_d^{\alpha_d}$, $0^0=1$, and $\partial ^\alpha=\partial^{\alpha_1}/{\partial x_1^{\alpha_1}}...\partial^{\alpha_d}/\partial x_d^{\alpha_d}$.
The length of $\alpha \in \NN_0^d$ is given by $|\alpha| = \sum_{i=1}^d \alpha_j$.
For $\gamma=(\gamma_1,\ldots,\gamma_d)\in\RR^d$ such that $-\gamma_j\not\in\NN$, $j=1,\ldots,d$, and $m\in\NN^d_0$, we use the abbreviation
\begin{equation*}
 \ds {\gamma\choose m}=
 \prod_{j=1}^d {\gamma_j\choose m_j}
= \prod_{j=1}^d \frac{\gamma_j (\gamma_j - 1) \dots (\gamma_j - m_j +1)}{m_j !}.
\end{equation*}

We also use the common notation for standard spaces of functions, distributions, and sequences.
For example, $ \ell^p $, $ p\geq 1$, denotes the Banach space of Lebesgue sequences, endowed with the  usual norm:
$$
\{a_n \}_{n \in \NN ^d _0} \in \ell^p \qquad \Longleftrightarrow \qquad \| a_n \|_{\ell^p}  =
( \sum_{n \in \NN ^d _0} |a_n|^p )^{1/p} <\infty,
$$
and $\displaystyle \{a_n \}_{n \in \NN ^d _0} \in \ell^\infty $ if and only if $
\|a_n \|_{ \ell^\infty } = \sup_{n \in \NN ^d _0} |a_n | < \infty$.

A  measurable function $f$ belongs to the Hilbert space  $L^2(\RR^d_+)$ if
$$ \|f\|_{L^2(\RR^d_+)} :=  \ds (\int_{\RR^d_+} |f(x)|^2 dx)^{1/2}<\infty,
$$ and the scalar product is given by
$ \ds
\langle f ,g \rangle = $ $ \int_{\RR^d_+} f(x) \overline{g (x)} dx, $ $ f,g \in L^2(\RR^d_+).$

By $\SSS(\mathbb{R}_+^d)$ we denote the space of all smooth functions on positive orthant, $f\in C^{\infty}(\RR^d_+)$ such that all of its derivatives $\partial^\alpha f$, $\alpha \in\NN^d_0$, extend to continuous functions on
$\overline{\RR^d_+}$, and
$$
\sup_{x\in\mathbb{R}^d_+}x^\beta |\partial^\alpha f(x)|<\infty,
\qquad \alpha, \beta \in\mathbb{N}_0^d,$$
and  $s$ denotes the space of rapidly decreasing sequences:
\begin{equation}
\label{eq:prostor-nizova-s}
\{a_n \}_{n \in \NN ^d _0} \in s \quad \Longleftrightarrow \quad
 \sum_{n \in \NN ^d _0} |a_n| |n|^{N} <\infty, \quad \forall N \in \NN.
\end{equation}

Let $I$ be a directed set and let  $\{X_{\alpha}\}_{\alpha\in I}$ be a family of normed spaces.
 By $ \bigcup_{\alpha\in I} X_{\alpha}$ we denote
 the union of $\{X_{\alpha}\}_{\alpha\in I}$, endowed with  inductive limit topology, and by
$ \bigcap_{\alpha\in I} X_{\alpha}$ we denote
 the  intersection of $\{X_{\alpha}\}_{\alpha\in I}$, endowed with   projective limit topology.

If $\mathcal{A}$ is a space of test functions, then $\mathcal{A}'$ denotes its dual space of (ultra)distributions, and $_{\mathcal{A}'} (\cdot, \cdot )_{\mathcal{A}} =  (\cdot, \cdot )$
denotes the dual pairing.

\subsection{Laguerre functions and Laguerre operator}

For $ n \in\mathbb{N}_0$ and $\gamma\geq 0$, the $n$-th Laguerre
polynomial of order $\gamma$ is defined by
$$
L_n ^\gamma(x)=\frac{x^{-\gamma}e^x}{n!}\frac{d^n}{dt^n}(e^{-x}x^{\gamma+n}),\quad
x \geq 0,
$$
and the $n$-th Laguerre function is defined by
$l_n ^\gamma (x)=L_n ^\gamma (x)e^{-x/2}$, $x \geq 0$.
The $d$-dimensional Laguerre polynomials and Laguerre
functions are given by tensor products:
$$
L_n^\gamma(x)=\prod_{k=1}^d L_{n_k}^\gamma(x_k) \qquad \text{ and} \qquad
l_n^\gamma(x)=\prod_{k=1}^d l_{n_k}^\gamma (x_k), \qquad
x \in \overline{\mathbb{R}^d _+}, \quad  n\in\NN^d_0.
$$
We will be particularly interested in the case $\gamma=0$, and write $L_n$ and $l_n$ instead of $L_n^0$ and $ l_n^0$, respectively. The Laguerre functions $\{l_n\}_{n\in\NN^d_0}$ form an orthonormal basis in $L^2(\mathbb{R}^d_+)$, cf. e.g. \cite{Thang, Wong}. Thus, every $f\in L^2(\mathbb{R}^d_+)$ can be expressed  via its Laguerre series expansion
$ \displaystyle
f = \sum_{n=0} ^{\infty} a_n(f) l_n, $ where
$a_n (f)$ denote the corresponding  Laguerre coefficients:
\begin{equation}
\label{eq:Laguerrecoeff}
a_n(f)=\int_{\mathbb{R}^d_+}f(x)l_{n}(x)dx, \qquad n\in\mathbb{N}^d_0.
\end{equation}

Moreover, the mapping $\iota:\mathcal{S}(\mathbb{R}^d_+)\rightarrow s$, given by
$\iota(f)=\{a_n(f)\}_{n\in\NN^d_0}$, is a topological isomorphism, see \cite[Theorem 3.1]{Sm}).

\par

The Laguerre functions are eigenfunctions of  the Laguerre operator $E$ given by
$$
E=-\sum_{j=1}^d\Big(  \frac{\partial}{\partial x_j}(x_j \frac{\partial}{\partial x_j})-\frac{x_j}{4}+\frac{1}{2} \Big)=-\sum_{j=1}^d \Big(x_j\frac{\partial ^2}{\partial x_j^2}+\frac{\partial}{\partial x_j}-\frac{x_j}{4}+\frac{1}{2}\Big),
$$
 and
$$
E ( l_n(x))=\Big(\sum_{i=1}^d  n_j\Big )\cdot l _n(x)=|n| \cdot l _n(x),\qquad
x \in \mathbb{R}^d _+, \quad  n\in\NN^d_0.
$$
see \cite[(11), p.188]{Ed}. It follows that
\begin{equation}\label{Eigen}
E^{N} ( l_n(x))=|n|^N l_{n}(x),
\quad x \in \mathbb{R}^d _+, \quad  N\in \mathbb N.
\end{equation}

\subsection{Hermite functions  and Hermite operator}

The Hermite polynomials $H_n$ and the corresponding Hermite functions $h_n$ are defined by
$$
H_n(x) =(-1)^n e^{x^2}\frac{d^n}{dx^n}(e^{-x^2}),
   \qquad  h_n (x)= (2^n n!\sqrt{\pi})^{-1/2} e^{-x^2/2}H_n(x),
    \quad
    x\in\mathbb{R}, \; n\in\NN_0,
$$
respectively. The $d$-dimensional Hermite polynomials $H_n$ and Hermite functions $h_n$ are given by
the tensor products
$$ H_n(x)=H_{n_1}(x_1)\ldots H_{n_d}(x_d),\qquad \mbox{and} \qquad
h_n(x)=h_{n_1}(x_1)\ldots h_{n_d}(x_d),\quad
x\in\mathbb{R}^d,\; n\in\N_0^d,
$$
respectively. Next, we introduce the Hermite operator:
$$
H= |x|^2-\Delta=\sum_{j=1}^d (x_j^2 - \frac{\partial^2}{\partial x^2_j}) \quad
x \in\mathbb{R}^d, \quad  n\in\N_0^d.
$$
Then the Hermite functions are eigenfunctions of $H$,
$ \displaystyle
H ( h_n(x))=(2|n|+d)h_n (x)$, $ x\in\RR^n$, $ n\in\N_0^d, $
see \cite[pp. 8]{Toft1}, and
\begin{equation}\label{EigenHermit}
H^{N} ( h_n(x))=  (2|n|+d)^N h_{n}(x)
 \qquad
x \in\mathbb{R}^d, \quad  N\in \mathbb N.
\end{equation}

\par

Hermite and Laguerre polynomials can be related as follows:
$$
H_{2n} (x) = (-1)^n 2^{2n} n! L_n ^{-1/2} (x^2), \qquad
H_{2n+1} (x) = (-1)^n 2^{2n} n! L_n ^{1/2} (x^2) x, \quad
x\in \mathbb{R} ^d,  n\in\N_0 ^d,
$$
see (1.1.52) and (1.1.53) in \cite{Thang}.

\subsection {The sequence spaces  $\ell_{\alpha}(\NN^d _0)$ and $\ell_{0,\alpha}(\NN^d _0)$}
\label{subsec:topology}

We utilize the notation from \cite{Toft1}.

\begin{defn} Let $\alpha>0 $,  $h>0$, and put $\vartheta_{h,\alpha}(n): = e^{h|n|^{1/(2\alpha)}}, $  $ n \in \NN^d_0$.  By $\ell^{p}_{[\vartheta_{h,\alpha}]}(\NN_0^d)$, $p \geq 1$,
we denote the space of all sequences $\{a_n\}_{n\in\NN^d_0}$, $a_n\in\mathbb{C}$, $n\in\N_0^d$,
which satisfy
\begin{equation} \label{eq:norm-sequence}
\|\{a_n\}_{n\in\mathbb \NN^d_0}\|_{\ell^{p}_{[\vartheta_{h,\alpha}]}}
:=\|\{|a_n|\cdot \vartheta_{h,\alpha}(n)\}_{n\in\NN^d_0}\|_{\ell^{p}}<\infty.
\end{equation}
\end{defn}

The space $\ell^{p}_{[\vartheta_{h,\alpha}]}(\NN_0^d)$ is a Banach space, with the
norm $ \| \cdot \|_{\ell^{p}_{[\vartheta_{h,\alpha}]}}$ given in \eqref{eq:norm-sequence}, and with the usual modification when $p = \infty$. Next, we introduce  locally convex spaces
$$
\ell_{\alpha}(\NN_0 ^d)=\bigcup_{h>0} \ell^{\infty}_{[\vartheta_{h,\alpha}]}(\NN_0^d)
\quad \text{and } \quad
\ell_{0,\alpha}(\NN_0^d) =\bigcap_{h>0} \ell^{\infty}_{[\vartheta_{h,\alpha}]}(\NN_0^d),
$$
in the sense of inductive and projective limit topology of $\ell^{\infty}_{[\vartheta_{h,\alpha}]}(\NN_0^d)$ with respect to $h>0$, respectively. In addition, we consider their dual (locally convex) spaces
$$
\ell_{\alpha}'(\NN_0 ^d)=\bigcap_{h>0} \ell^{\infty}_{[1/\vartheta_{h,\alpha}]}(\NN_0^d)
\quad \text{and } \quad
\ell_{0,\alpha}'(\NN_0^d) =\bigcup_{h>0} \ell^{\infty}_{[1/\vartheta_{h,\alpha}]}(\NN_0^d),
$$
in the sense of projective and inductive limit topology of $\ell^{\infty}_{[1/\vartheta_{h,\alpha}]}(\NN_0^d)$ with respect to $h>0$, respectively.
When $\alpha=0 $,  $\ell_0(\NN^d_0)$ denotes the space of all sequences $\{a_n\}_{n\in\NN^d_0}$ such that $a_n\not=0$ for only finitely many $n \in \NN_0^d$, and $\ell_{0,0}(\NN^d_0)=\{0\}.$

Note that a different notation is used in  \cite{SSB}. The reader may immediately see that for  $\alpha \geq 1$, the sequence space $s^{\alpha,a} (\NN_0^d)$ defined in \cite{SSB}
is the same as $\ell^{\infty}_{[\vartheta_{\ln a,\alpha/2}]}(\NN_0^d)$,
$a>1$, and consequently
$ s^\alpha (\NN_0^d) = \ell^{\infty}_{\alpha/2} (\NN_0^d)$.

\begin{rem}\label{rem:norme}
Observe that different choices of $p \in [1,\infty]$ in \eqref{eq:norm-sequence}
give rise to equivalent families of norms for
$ \ell_{\alpha}(\NN_0 ^d)$ and $\ell_{0,\alpha}(\NN_0^d) $.
We are particularly interested in the case  $p\in \{2,\infty\}$, i.e.
$\ds \ell^{\infty}_{[\vartheta]}(\NN_0^d)= \ell^{2}_{[\vartheta]}(\NN_0^d)$. Thus, if
$\{a_n\}_{n\in\mathbb \NN^d_0} \in \ell^{\infty}_{[\vartheta_{h,\alpha}]}(\NN_0^d) $ then for any given $h>0$ there exist $h_1, h_2 >0$ and $C_1, C_2>0$ such that
$$ C_1 \|\{a_n\}_{n\in\mathbb \NN^d_0\}}\|_{\ell^{\infty}_{[\vartheta_{h_1,\alpha}]}(\NN_0^d)}
\leq
\Big(\sum_{n\in\NN^d_0}|a_n|^2 \vartheta_{h,\alpha}(n)\Big)^{1/2}
\leq
C_2  \|\{a_n\}_{n\in\mathbb \NN^d_0}\|_{\ell^{\infty}_{[\vartheta_{h_2,\alpha}]}(\NN_0^d)}.$$
\end{rem}

Regarding the topological structure of $\ell_{\alpha}(\NN^d _0)$ and $\ell_{0,\alpha}(\NN^d _0)$, we have the following.

\begin{prop} \label{lm:nizovi-potapanje}
Let there be given $\alpha>0$, and $h_1>h_2>0$. Then the canonical inclusion
\begin{equation} \label{eq:inclusion}
    \ell^{\infty}_{[\vartheta_{h_1,\alpha}]}(\NN_0^d)\rightarrow \ell^{\infty}_{[\vartheta_{h_2,\alpha}]}(\NN_0^d)
\end{equation}
is nuclear.
\end{prop}

\begin{proof}
The proof is similar to the proof of \cite[Proposition 2.1]{SSB}, with slight modifications, and we present it here for the sake of completeness.
We will use the fact that a mapping is nuclear if it can be represented as a
composition of two quasinuclear mappings, cf. \cite{Pietsch}.
Since the  inclusion \eqref{eq:inclusion}
can be considered as a composition of inclusions $ \ell^{\infty}_{[\vartheta_{h_1,\alpha}]}(\NN_0^d)\rightarrow \ell^{\infty}_{[\vartheta_{h_3,\alpha}]}(\NN_0^d)$
and
$ \ell^{\infty}_{[\vartheta_{h_3,\alpha}]}(\NN_0^d)\rightarrow \ell^{\infty}_{[\vartheta_{h_2,\alpha}]}(\NN_0^d)$, for some/any $h_3\in (h_2,h_1)$, it is enough to prove the quasinuclearity of the inclusion $I_{h_1}^{h_3}:\ell^{\infty}_{[\vartheta_{h_1,\alpha}]}(\NN_0^d)\rightarrow \ell^{\infty}_{[\vartheta_{h_3,\alpha}]}(\NN_0^d)$.

\par

By $(\ell^{\infty}_{[\vartheta_{h_1,\alpha}]}(\NN_0^d))'$ we denote the dual space of
$\ell^{\infty}_{[\vartheta_{h_1,\alpha}]} (\NN_0^d)$.  Let there be given
 $\{a_n\}_{n\in\NN^d_0}\in  \ell^{\infty}_{[\vartheta_{h_1,\alpha}]} (\NN_0^d)$ and let
 $e_m\in (\ell^{\infty}_{[\vartheta_{h_1,\alpha}]}(\NN_0^d))'$ be given by
$$
e_m\big(\{a_n\}_{n\in\NN^d_0}\big)=a_m e^{h_3 |m|^{1/(2\alpha)}}, \qquad m\in\NN^d_0.
$$
Then clearly
$\|e_m\|_{(\ell^{\infty}_{[\vartheta_{h_1,\alpha}]})'}\leq C e^{(h_3-h_1)|m|^{1/(2\alpha)}}$ for some $C>0$,
and therefore
$$\sum_{m\in\NN^d_0} \|e_m\|_{(\ell^{\infty}_{[\vartheta_{h_1,\alpha}]})'}<\infty, $$
since $h_3<h_1$. Now we have
\begin{multline*} 
\|I_{h_1}^{h_3}(\{a_n\}_{n\in\NN^d_0})\|_{\ell^{\infty}_{[\vartheta_{h_3,\alpha}]}}=\|\{a_n\}_{n\in\NN^d_0}\|_{\ell^{\infty}_{[\vartheta_{h_3,\alpha}]}} \\
=\sup_{m\in\NN^d_0} |a_m| e^{h_3|m|^{1/(2\alpha)}}=\sup_{m\in\NN^d_0} \Big|e_m\big(\{a_n\}_{n\in\NN^d_0}\big)\Big|  \leq \sum_{m\in\NN^d_0} \big|e_m(\{a_n\}_{n\in\NN^d_0})\big|,
\end{multline*}
for every $\{a_n\}_{n\in\NN^d_0}\in  \ell^{\infty}_{[\vartheta_{h_1,\alpha}]} (\NN_0^d)$. Then
by  \cite[Definition 3.2.3]{Pietsch} it follows that $I_{h_1}^{h_3}$ is
quasinuclear. The same is true for $I_{h_3}^{h_2}: \ell^{\infty}_{[\vartheta_{h_3,\alpha}]}(\NN_0^d)\rightarrow \ell^{\infty}_{[\vartheta_{h_2,\alpha}]}(\NN_0^d)$,
so that \eqref{eq:inclusion} is nuclear.
\end{proof}

By Proposition \ref{lm:nizovi-potapanje} and \cite[Theorem 50.1, p.511]{Tr}
it follows that $\ell_{\alpha}(\NN_0^d)$ is a $(DFN)-$space (dual of a Frech\'et nuclear space),
and that $\ell_{0,\alpha}(\NN_0^d)$ is an $(FN)-$space (Frech\'et nuclear space).

Thus $\ell_{\alpha}'(\NN_0^d)$ is a $(FN)-$space and $\ell_{0,\alpha}'(\NN_0^d)$ is an $(DFN)-$space (cf. \cite[Proposition 2.2]{SSB} and \cite[Proposition 1.1]{SSBb}).

Note that $\ds \ell_{0}(\NN_0^d)=\bigcup_{n\in\NN} E_n$,
where $E_n=\{\{a_k\}_{k\in \NN_0}, k=0\ \mbox{when}\ k>n\}$.
Since the canonical inclusion $id : E_n \rightarrow E_{n+1}$ is nuclear,
it  follows that $\ell_{0}(\NN_0^d)$ is a  $(DFN)-$space.

\subsection{Test spaces of formal Laguerre and Hermite series} \label{subsec:formal}


Following \cite{Toft1} (where the spaces of formal Hermite series expansions were considered)
by $\mathcal G_{\alpha}(\RR_{+}^d)$ and $\mathcal G_{0,\alpha}(\RR_{+}^d)$ we denote the spaces of formal Laguerre series  expansions
$f=\sum_{n\in\NN_0^d} a_{n} l_n$ that correspond to sequences $\{a_n\}_{n\in\NN^d_0} \in \ell_{\alpha/2}(\NN_0^d)$ and  $\{a_n\}_{n\in\NN^d_0} \in \ell_{0,\alpha/2}(\NN_0^d)$, $\alpha \geq 0$, respectively. By mappings
\begin{eqnarray} \label{eq:mappingT}
    T: \ell_{\alpha/2}(\NN_0^d) \rightarrow\mathcal G_{\alpha}(\RR_{+}^d), \qquad T(\{a_n\}_{n\in\NN^d_0})=\sum_{n\in\NN_0^d} a_n l_n, \\
    T: \ell_{0,\alpha/2}(\NN_0^d) \rightarrow\mathcal G_{0,\alpha}(\RR_{+}^d), \qquad T(\{a_n\}_{n\in\NN^d_0})=\sum_{n\in\NN_0^d} a_n l_n,
   \end{eqnarray}
topologies  in $\mathcal G_{\alpha}(\RR_{+}^d)$ and $\mathcal G_{0,\alpha}(\RR_{+}^d)$  are inhereted from the ones in $\ell_{\alpha/2}(\NN_0^d)$ and  $\{a_n\}_{n\in\NN^d_0} \in \ell_{0,\alpha/2}(\NN_0^d)$, $\alpha \geq 0$, respectively.


The spaces of formal Hermite series  expansions
$f=\sum_{n\in\NN_0^d} a_{n} h_n$ that correspond to the sequence $\{a_n\}_{n\in\NN^d_0} \in \ell_{\alpha}(\NN_0^d)$, or  $\{a_n\}_{n\in\NN^d_0} \in \ell_{0,\alpha}(\NN_0^d)$, $\alpha \geq 0$, are denoted by $\mathcal H_{\alpha}(\RR ^d)$ and $\mathcal H_{0,\alpha}(\RR ^d)$, respectively, and the natural mappings
\begin{eqnarray} \label{eq:mappingTHermite}
    T_h: \ell_{\alpha}(\NN_0^d) \rightarrow\mathcal H_{\alpha}(\RR ^d), \qquad T_h (\{a_n\}_{n\in\NN^d_0})=\sum_{n\in\NN_0^d} a_n h_n, \\
        T_h: \ell_{0,\alpha}(\NN_0^d) \rightarrow\mathcal H_{0,\alpha}(\RR ^d), \qquad T_h (\{a_n\}_{n\in\NN^d_0})=\sum_{n\in\NN_0^d} a_n h_n,
         \label{eq:mappingTHermite2}
                   \end{eqnarray}
induce  topologies on the target spaces.

\subsection{Pilipovi\'c and Gelfand-Shilov spaces on $\RR ^d $}
Pilipovi\'{c} spaces are defined by using the iterates of Hermite operators  as follows.

We say that  $f\in C^{\infty}(\RR^d)$ belongs to
the (global) Pilipovi\'{c} spaces of Roumieu type $\mathbfcal{S}^{\alpha} _{\alpha}(\RR^d)$
(of Beurling type $\mathbf{\Sigma} ^{\alpha} _{\alpha}(\RR^d)$), $ \alpha >0$,
if there exist $h>0$ and $C>0$ (for every $h>0$  there exists $C=C_h>0$)
such that the following estimate holds:
$$
\|H^N f\|_{L^{\infty}(\RR^d)}\leq C h^N N!^{2\alpha}.
$$

From \cite[Definition 3.1]{Toft1} and \cite[Theorem 5.2]{Toft1}
it follows that  the spaces of formal Hermite expansions
$\mathcal H_{\alpha}(\RR ^d)$ and $\mathcal H_{0,\alpha}(\RR ^d)$,
coincide with Pilipovi\'{c} spaces $\mathbfcal{S}^{\alpha} _{\alpha}(\RR^d)$ and  $\mathbf{\Sigma} ^{\alpha} _{\alpha}(\RR^d)$, $\alpha >0$, respectively.
For that reason, from now on we will use the notation
$\mathcal H_{\alpha}(\RR ^d)$ and $\mathcal H_{0,\alpha}(\RR ^d)$,
instead of  $\mathbfcal{S}^{\alpha} _{\alpha}(\RR^d)$ and  $\mathbf{\Sigma} ^{\alpha} _{\alpha}(\RR^d)$,
respectively.

Next we recall the definition of Gelfand-Shilov spaces.

Let there be given $\alpha,A>0$. We denote by
$\SSS^{\alpha,A}_{\alpha,A}(\RR^d)$  the Banach space of all
$\varphi\in C^{\infty}(\RR^d)$ such that
$$
\Vert\varphi\Vert_{\SSS^{\alpha,A}_{\alpha,A}}: =\sup_{p,k\in\NN^d_0}\frac{\left\|x^k
\partial^p\varphi(x)\right\|_{L^2(\RR^d)}}{
A^{|p+k|} k!^{\alpha} p!^{\alpha}}<\infty,
$$
and $\Vert\cdot \Vert_{\SSS^{\alpha,A}_{\alpha,A}}$ is the norm in
$\SSS^{\alpha,A}_{\alpha,A}(\RR^d)$.

The Gelfand-Shilov space (of Roumieu type)
$\SSS^{\alpha}_{\alpha}(\RR^d)$, $\alpha> 0$ (cf. \cite[Chapter]{GS2}) is defined by
$$\SSS^{\alpha}_{\alpha}(\RR^d)= \bigcup_{A>0}
\SSS^{\alpha,A}_{\alpha,A}(\RR^d),$$
and it is endowed with inductive limit topology.   $\SSS^{\alpha}_{\alpha}(\RR^d)$ is nontrivial if and only if $\alpha\geq 1/2$, cf. \cite[Chapter IV, Section 8.2]{GS2}. Note that $h_n\in\SSS^{1/2}_{1/2}(\RR^d)$, $n\in\NN^d_0$, and Gelfand-Shilov spaces obviously increase with respect to $\alpha$.

The Gelfand-Shilov distribution space of Roumieu type $(\SSS^{\alpha}_{\alpha})'(\RR^d)$ is the dual space of $\SSS^{\alpha}_{\alpha}(\RR^d)$, $\alpha > 0$, i.e. it is given by
$$(\SSS^{\alpha}_{\alpha})'(\RR^d)= \bigcap_{A>0}
(\SSS^{\alpha,A}_{\alpha,A})'(\RR^d),$$
and endowed with projective limit topology.

The spaces $\SSS^{\alpha}_{\alpha}(\RR^d)$ and $(\SSS^{\alpha}_{\alpha})'(\RR^d)$ can be characterized through Hermite expansions as follows.

\begin{prop}\label{char S and S'}
Let $\alpha\geq 1/2$. The mapping
$\SSS^{\alpha}_{\alpha}(\RR^d)\rightarrow \ell_{\alpha}(\N_0^d)$,
$f\mapsto \{\langle f,h_n\rangle\}_{n\in\NN^d_0}$, is a
topological isomorphism. Moreover,
$\sum_{n\in\NN^d_0}\langle f,h_n\rangle h_n$ converges absolutely to $f$ in $\SSS^{\alpha}_{\alpha}(\RR^d)$.

The mapping $(\SSS^{\alpha}_{\alpha})'(\RR^d)\rightarrow
\ell'_{\alpha}(\N_0^d)$, $u\mapsto \{\langle
u,h_n\rangle\}_{n\in\NN^d_0}$, is a topological isomorphism. Moreover,
$\sum_{n\in\NN^d_0}\langle
u,h_n\rangle h_n$ converges absolutely to $u$ in
$(\SSS^{\alpha}_{\alpha})'(\RR^d)$.
\end{prop}

 Proposition \ref{char S and S'} is a special case of  Theorem 3.4 and Corollary 3.5 in
 \cite{lan}.

We also consider
 Gelfand-Shilov spaces of Beurling type
$$\Sigma ^{\alpha}_{\alpha}(\RR^d)= \bigcap_{A>0}
\SSS^{\alpha,A}_{\alpha,A}(\RR^d),$$
endowed with projective limit topology, and their dual spaces
of Gelfand-Shilov distributions of Beurling type
$ (\Sigma ^{\alpha}_{\alpha})' (\RR^d)$.

The space $\Sigma ^{\alpha}_{\alpha}(\RR^d)$ is nontrivial if and only if $ \alpha > 1/2$,
and
$$
\Sigma ^{\alpha}_{\alpha}(\RR^d) \subset \SSS^{\alpha}_{\alpha}(\RR^d)
\subset \Sigma ^{\beta}_{\beta}(\RR^d) \subset \SSS^{\beta}_{\beta}(\RR^d), \qquad
\frac{1}{2} < \alpha < \beta.
$$

An analogue of Proposition \ref{char S and S'}
can be formulated for $\Sigma ^{\alpha}_{\alpha}(\RR^d)$
and $( \Sigma ^{\alpha}_{\alpha})' (\RR^d)$, $ \alpha > 1/2$,
and we leave details for the reader.

The following result is Theorem 5.2 in \cite{Toft1}, see also \cite{SP1988}.

\begin{prop}\label{char H and S}
Let $\alpha > 1/2$. Then,
$$
\mathcal H_{\alpha}(\RR ^d) = \SSS^{\alpha}_{\alpha}(\RR^d)
\quad \text{and} \quad
\mathcal H_{0,\alpha}(\RR ^d) = \Sigma ^{\alpha}_{\alpha}(\RR^d),
$$
In addition, $ \mathcal H_{1/2}(\RR ^d) = \SSS^{1/2}_{1/2}(\RR^d)$ and
$ \mathcal H_{0, 1/2}(\RR ^d) \neq \Sigma ^{1/2}_{1/2}(\RR^d) = \{ 0\}$.
\end{prop}



Since  Pilipovi\'c spaces are non-trivial even when $
\alpha \in (0, 1/2)$, they can be considered as a
natural extension of  Gelfand-Shilov spaces when considering ultradifferentiable functions of ultra-rapid decay.

\subsection{$G$-type spaces  on $\RR^d _+$}
We end this section with the definition  of $G$-type spaces and their distribution spaces, as it is given in \cite{SSB, SSBb, SSSB}.
In addition, we recall the Laguerre series expansions in these spaces.

\par


Let $A>0$. By $g^{\alpha,A}_{\alpha,A}(\RR^d_+)$, we denote the
space of all $f\in\SSS(\RR^d_+)$ such that
\begin{equation} \label{seminorme-bez-dodatka}
\sup_{p,k\in\NN^d_0}\frac{\|x^{(p+k)/2} \partial^pf(x)\|_{L^2(\RR^d_+)}}
{A^{|p+k|}k^{(\alpha/2)k}p^{(\alpha/2)p}}<\infty.
\end{equation}
With the seminorms
\begin{equation} \label{seminorme-sa-dodatkom}
\sigma_{A,j}^{\alpha, \alpha} (f)=\sup_{p,k\in\NN^d_0}\frac{\|x^{(p+k)/2} \partial^pf(x)\|_{L^2(\RR^d_+)}}
{A^{|p+k|}k^{(\alpha/2)k}p^{(\alpha/2)p}}+\sup_{\substack{|p|\leq
j\\ |k|\leq j}} \sup_{x\in\RR^d_+}|x^k \partial ^p f(x)|,\,\, j\in\NN_0,
\end{equation}
it becomes a Frech\'et space ($(F)$-space).

The spaces $G^{\alpha}_{\alpha}(\RR^d_+)$ and
$g^{\alpha}_{\alpha}(\RR^d_+)$ are then defined as union, respectively intersection of
$g^{\alpha,A}_{\alpha,A}(\RR^d_+)$ with respect to $A$:
\begin{equation}
\label{eq:velikoimaloge}
G^{\alpha}_{\alpha}(\RR^d_+)=
\bigcup_{A>0}
g^{\alpha,A}_{\alpha,A}(\RR^d_+), \qquad
g^{\alpha}_{\alpha}(\RR^d_+)=
\bigcap_{A>0}
g^{\alpha,A}_{\alpha,A}(\RR^d_+).
\end{equation}
Then $G^{\alpha}_{\alpha}(\RR^d_+)$ and
$g^{\alpha}_{\alpha}(\RR^d_+)$ are endowed with inductive, respectively projective limit topology.
We refer to  $G^{\alpha}_{\alpha}(\RR^d_+)$, $\alpha\geq 1$, and $g^{\alpha}_{\alpha}(\RR^d_+)$, $\alpha> 1$, as  $G$-type spaces of Roumieu and Beurling type,  respectively.


\par

$G$-type distribution spaces $(G^{\alpha}_{\alpha})'(\RR^d_+)$ and
$(g^{\alpha}_{\alpha})'(\RR^d_+)$ are then given by
\begin{equation}
\label{eq:velikoimalogeprim}
(G^{\alpha}_{\alpha})'(\RR^d_+)=
\bigcap_{A>0}
(g^{\alpha,A}_{\alpha,A})'(\RR^d_+), \quad
\alpha \geq 1, \quad
(g^{\alpha}_{\alpha})'(\RR^d_+)=
\bigcup_{A>0}
(g^{\alpha,A}_{\alpha,A})'(\RR^d_+),
 \quad \alpha > 1,
\end{equation}
and they are endowed with projective, respectively, inductive limit topology.
Here $(g^{\alpha,A}_{\alpha,A})'(\RR^d_+)$ denotes the dual space of 
$g^{\alpha,A}_{\alpha,A}(\RR^d_+)$.

Next, we  summarize results from \cite[Theorem 3.6]{D1}, \cite[Theorem 5.7]{SSB} and \cite[Theorem 3.7]{SSBb} as follows:

\begin{thm} \label{D Th3.6}
Let  $f\in\SSS(\RR^d_+)$.
Then the following is true.
\begin{enumerate}
\item[$(i)$] Let $ \alpha \geq 1$. The mapping $G_\alpha^\alpha(\mathbb{R}^d_+)\rightarrow \ell _{\alpha/2}(\NN_0^d)$, $f\mapsto \{\langle f,l_n\rangle\}_{n\in\NN^d_0}$, is a
topological isomorphism. Moreover,
$\sum_{n\in\NN^d_0}\langle f,l_n\rangle l_n$ converges absolutely to $f$ in $G_\alpha^\alpha(\mathbb{R}^d_+)$.

\item[$(ii)$] Let $ \alpha > 1$. The mapping $g_\alpha^\alpha(\mathbb{R}^d_+)\rightarrow \ell _{0,\alpha/2}(\NN_0^d)$, $f\mapsto \{\langle f,l_n\rangle\}_{n\in\NN^d_0}$, is a
topological isomorphism. Moreover,
$\sum_{n\in\NN^d_0}\langle f,l_n\rangle l_n$ converges absolutely to $f$ in $g_\alpha^\alpha(\mathbb{R}^d_+)$.

\item[$(iii)$] Let $ \alpha \geq 1$. The mapping $(G_\alpha^\alpha)'(\mathbb{R}^d_+)\rightarrow \ell '_{\alpha/2}(\NN_0^d)$, $u\mapsto \{\langle u,l_n\rangle\}_{n\in\NN^d_0}$, is a
topological isomorphism. Moreover,
$\sum_{n\in\NN^d_0}\langle u,l_n\rangle l_n$ converges absolutely to $u$ in $(G_\alpha^\alpha)'(\mathbb{R}^d_+)$.

\item[$(iv)$]  Let $ \alpha > 1$. The mapping $(g_\alpha^\alpha)'(\mathbb{R}^d_+)\rightarrow \ell '_{0,\alpha/2}(\NN_0^d)$, $u\mapsto \{\langle u,l_n\rangle\}_{n\in\NN^d_0}$, is a
topological isomorphism. Moreover,
$\sum_{n\in\NN^d_0}\langle u,l_n\rangle l_n$ converges absolutely to $u$ in $(g_\alpha^\alpha)'(\mathbb{R}^d_+)$.
\end{enumerate}
\end{thm}

\par

Theorem \ref{D Th3.6} implies that
$\mathcal{G}_{0,\alpha}(\RR^d_+)=g_\alpha^\alpha(\RR^d_+)$, when $\alpha>1$
and $\mathcal{G}_\alpha(\RR^d_+)=G_\alpha^\alpha(\RR^d_+)$, when $\alpha\geq 1$.

We refer to \cite[Corollary 3.9]{D4} for the triviality of
$ G_\alpha^\alpha(\RR^d_+)$, when $\alpha < 1$,
and we end this subsection with a result on the  triviality of  $g_1 ^1 (\RR^d_+) $. Although it seems to be known, we could not find an explicit proof in the existing literature.

We first prove a technical lemma.

\begin{lem}  \label{tehnickalema}   Let  $ g^1_1(\RR^d_+) = \bigcap_{A>0}
g^{1,A}_{1,A}(\RR^d_+)$, where $g^{1,A}_{1,A}(\RR^d_+)$ is given by \eqref{seminorme-bez-dodatka} with $\alpha = 1$. Then the following are equivalent.
\begin{itemize}
\item[$(i)$] $f\in g^1_1(\RR^d_+)$.
\item[$(ii)$]  For every  $h>0$ there exists $C=C_h>0$ such that
\begin{equation}\label{gsc}
\|x^k \partial^p f(x)\|_{L^2(\RR^d_+)} \leq C h^{|p+k|} |k|!, \qquad k,p\in \NN_0^d.
\end{equation}
\item[$(iii)$]  For every  $h>0$ there exists $C=C_h>0$ such that
\begin{equation}\label{zvezda}
\sup_{x\in \RR^d_{+}} |x^k \partial^p f(x)|\leq C h^{|p+k|} |k|!, \quad k,p \in \NN_0^d.
\end{equation}
\end{itemize}

\end{lem}

\begin{proof}
We consider the case $d=1$ for the sake of simplicity.

$(i) \Rightarrow (ii)$ Let there be given $k,p\in \NN_0$.
By using the estimates similar to the ones given in the proof of \cite[Theorem 2.4]{D4}
we obtain
$$\|x^k  f^{(p)}(x)\|^2_2 \leq \frac{(2k)!}{(2k+2p)!} \sum_{m=0}^{2p}{2p\choose m} \|x^{\frac{p+m+2k}{2}}f^{(p+m)}(x)\|_{L^2(\RR^d_+)} \cdot \|x^{\frac{3p-m+2k}{2}} f^{(3p-m)}(x)\|_{L^2(\RR^d_+)}
$$
$$\leq \frac{1}{(2p)!} C^2 A^{4(p+k)} (2k)!
\cdot \sum_{m=0}^{2p}{2p\choose m} \big((p+m)!  (3p-m)! \big)^{1/2}$$
$$\leq \frac{\sqrt{(4p)!}}{(2p)!} C^2 A^{4 (p+k)} 2^{4k} k!^2 
\leq \big( C (4 A^2)^{(p+k)} k! \big)^2
$$
where $A>0$ can be chosen arbitrary since $f\in g^1_1(\RR_+)$,
and $C>0$ depends on $A$.
Thus, \eqref{gsc} holds.

$(ii) \Rightarrow (iii)$ We have
$\ds x^k f^{(p)}(x)=\int_{-\infty}^{x} \left(  t^k f^{(p)}(t)\right)' dt$,
$x\in \RR_{+}$,  for any  $k,p\in \NN_0$ so that
\begin{multline*}
\left| x^k f^{(p)}(x)\right|\leq  k \left|\int_{-\infty}^{x} t^{k-1} f^{(p)}(t) dt\right|+ \left|\int_{-\infty}^{x} t^{k} f^{(p+1)}(t) dt\right| \\
\leq k\left(\int_{\RR} \frac{dt}{(1+|t|)^2}\right)^{1/2}\cdot \left(\int_{\RR} {(1+|t|)^{2k}}\left(f^{(p)}(t)\right)^2\right)^{1/2} \\
+ \left( \int_{\RR} \frac{dt}{(1+|t|)^2}  \right)^{1/2}\cdot \left( (1+|t|)^{2k+2} \left(f^{(p+1)}(t)\right)^2 \right)^{1/2}  \\
\leq C_1 \left(  k\| x^{k} f^{(p)}(x)\|_2 +\| x^{k+1} f^{(p+1)}(x)\|_2 \right),  \quad x\in \RR^{+},
\end{multline*}
for a suitable $C_1>0$. Thus  (\ref{gsc}) implies
(\ref{zvezda}).

The implication $(iii) \Rightarrow (ii)$  is obvious.

$(ii) \Rightarrow (i)$ We slightly modify the proof of  \cite[Theorem 2.4 c)]{D4}. 
Assume that \eqref{gsc} holds. Then for any given  $k,p\in \NN_0$ we have
$$\|x^{\frac{p+k}{2}}  f^{(p)}(x)\|^2_2 \leq C h_1 ^{\frac{p+k}{2} +p} (\frac{p+k}{2})^{\frac{p+k}{2}}
\leq
C h_1^{\frac{3p}{2} + \frac{k}{2}} 2^{-\frac{p+k}{2}}
(p+k)^{\frac{p+k}{2}}
\leq C h_1^{\frac{3p}{2} + \frac{k}{2}} (2e)^{\frac{p+k}{2}}
\sqrt{p!} \sqrt{k!},
$$
for any $h_1>0$. Let there be given $h>0$. Then, by choosing $ l < \min\{\frac{h^2}{2e}, 1\}$
we obtain
$$\|x^{\frac{p+k}{2}}  f^{(p)}(x)\|^2_2 \leq C h^{p+q} \sqrt{p!} \sqrt{k!},
$$
where $C>0$  depends on $h$. Thus $f \in  g^1_1(\RR^d_+)$.
\end{proof}

    Now we are able to prove the following.

    \begin{thm} \label{triviality}
The space $g^\alpha _\alpha (\RR^d_{+})$ is trivial,  i.e. $g^\alpha _\alpha (\RR^d_{+})=\{0\}$,
for every $ 0< \alpha \leq 1$.
    \end{thm}

\begin{proof} It is enough to prove $g^1 _1 (\RR^d_{+})=\{0\}$.
    We give the proof for $d=1$. For $d\geq 2$ we may use e.g. the kernel theorem, 
    \cite[Theorem 4.4]{SSBb}.

It is enough to verify that we can choose $h>0$ in \eqref{zvezda} such that $\displaystyle \frac{1}{\sqrt{e}}$,
because then  \cite[Theorem 5.2]{D3} implies that $f \equiv 0$ if in \eqref{zvezda} $h>0$ can be chosen so that $h < \displaystyle \frac{1}{\sqrt{e}}$. Indeed, Lemma \ref{tehnickalema} gives that if $f \in g^1 _1 (\RR^d_{+})$ then
we can choose $h>0$  arbitrary small  in  \eqref{zvezda}. Thus, we conclude that $f \equiv 0.$
\end{proof}

\section{Pilipovi\'{c} spaces on $\RR^d_+$} \label{sec3}

We start with a simple observation.

\begin{lem} \label{lm:S+}
Let $f\in  C^{\infty}(\RR^d_{+})$  such that $E^N f\in L^2(\RR_{+}^d)$ for every $N\in\NN_0$. Then $f\in \mathcal S(\RR_{+}^d)$.
\end{lem}

\begin{proof}
Let us consider the  Laguerre expansion of $f$, namely
$f=\sum_{n\in\NN^d_0} a_n(f) l_n$.
By (\ref{Eigen}) it follows that
$E^N f = \sum_{n\in\NN^d_{0}} |n|^N a_n(f) l_n$. Since   $E^N f\in L^2(\RR_{+}^d)$,
we get $\sum_{n\in\NN^d_0} (a_n(f))^2 |n|^{2N} $ $<\infty$, so that
for every $N\in\NN_0$ there exists $C_N>0$ such that $\ds |a_n(f)|\leq \frac{C_N} {(1+|n|)^{N}}$.

Let $k,p\in\NN_0^d$. By using the estimate
\begin{equation}\label{lagers}
\sup_{x\in\RR^d_{+}}|x^k \partial^p l_n(x)|\leq C_{p,k} (1+|n|)^{|p+k|}
\end{equation}
(see (3.1) in \cite{Sm}) we have
$$
|x^k \partial^p f(x)|\leq \sum_{n\in\NN^d_0} |a_n(f)|\cdot |x^k \partial^p l_n(x)|\leq C_{p,k} C_N\sum_{n\in\NN^d_0}\frac{(1+|n|)^{|p+k|}}{(1+|n|)^N}, \quad x\in\RR^d_{+},
$$
for any $N\in\NN_0$.
Since $N$ can be chosen arbitrary large, it follows that
$ \ds \sup_{x \in \RR^d_{+}}|x^k \partial^p f(x)| <\infty,$
so that $f\in \mathcal S(\RR^d_{+})$.
\end{proof}


We proceed with the definition of  $P-$spaces on positive orthants.

\begin{defn} \label{def:PilipovicOrthant}
Let $h>0$, and $\alpha>0$. Then a smooth functions $ f \in C^{\infty} (\RR^d_+)$ belongs to the Banach space $ \mathbf G^{\alpha,h}_{\alpha,h}(\RR^d_+)$ with the norm $\eta_h^\alpha$  given by
\begin{equation}\label{eta}
\eta_h^\alpha(f) : =\sup_{N\in\NN}\frac{\Vert E^N f\Vert_{L^2(\RR^d_+)}}{h^{N} N!^\alpha}<\infty.
\end{equation}

The Pilipovi\'c spaces ($P-$spaces) over $\RR^d_+$, $\mathbf{G}^{\alpha}_{\alpha}(\RR^d_+)$ and  $\mathbf{g}^{\alpha}_{\alpha}(\RR^d_+)$ of Roumieu and Beurling type respectively, are defined by
\begin{equation*}
 \mathbf G^{\alpha}_{\alpha}(\RR^d_+)=
\bigcup_{h>0}  \mathbf
G^{\alpha,h}_{\alpha,h}(\RR^d_+), \qquad \text{and} \qquad
\mathbf{g}^{\alpha}_{\alpha}(\RR^d_+)=
\bigcap_{h>0}  \mathbf
G^{\alpha,h}_{\alpha,h}(\RR^d_+).
\end{equation*}
\end{defn}
The spaces $ \mathbf G^{\alpha}_{\alpha}(\RR^d_+)$ and
$ \mathbf g^{\alpha}_{\alpha}(\RR^d_+)$ are endowed with inductive, respectively projective limit topology.

By Lemma \ref{lm:S+} it follows that  $\mathbf G^{\alpha,h}_{\alpha,h}(\RR^d_+) \subset \mathcal S(\RR^d_{+})$, and
$$
\eta_h^\alpha(f) < \infty \quad \Longleftrightarrow \quad
\eta_h^\alpha(f) +  \sup_{\substack{|p|\leq
j\\ |k|\leq j}} \sup_{x\in\RR^d_+}|x^k  \partial^p f(x)| < \infty,\,\, j\in\NN_0.
$$
Clearly, if $0< \alpha<\beta$, $\mathbf G^{\alpha,h}_{\alpha,h}(\RR^d_+)$ is continuously injected into $\mathbf G^{\beta,h}_{\beta,h}(\RR^d_+)$.

The space  $\mathbf G^{\alpha,h}_{\alpha,h}(\RR^d_+)$, $ \alpha>0$, $h>0$, is  non-trivial. Indeed, Laguerre functions $l_n$,  $n\in\NN^d_0$,  satisfy (\ref{eta}) for every $h>0$
since
$$\sup_{N\in\NN}\frac{\Vert E^N l_n\Vert_{L^2(\RR^d_+)}}{h^{N} N!^{\alpha}}=\sup_{N\in\NN}\frac{(\frac{|n|}{h})^N}{N!^{\alpha}}<\infty,$$
so that  $\mathbf G^{\alpha,h}_{\alpha,h}(\RR^d_+)$ contains finite linear combinations of
Laguerre functions.

$\mathbf G^{\alpha}_{\alpha}(\RR^d_+)$ and $\mathbf{g}^{\alpha}_{\alpha}(\RR^d_+)$ are continuously injected into $\mathcal S(\RR^d_+)$, and
the following embeddings are obvious:
\begin{equation}\label{poredak}
\mathbf{g}^{\alpha}_{\alpha}(\RR^d_+)\subset \mathbf{G}^{\alpha}_{\alpha}(\RR^d_+)\subset  \mathbf{g}^{\beta}_{\beta}(\RR^d_+) \subset \mathbf{G}^{\beta}_{\beta}(\RR^d_+),
\qquad 0<\alpha < \beta.
\end{equation}

\par

The corresponding dual spaces are denoted by $(\mathbf{G}^{\alpha}_{\alpha})'(\RR^d_+)$ and  $(\mathbf{g}^{\alpha}_{\alpha})'(\RR^d_+)$.

The main result of this section is to characterize $P-$spaces and their dual spaces on positive orthants in the spirit of Theorem  \ref{D Th3.6}.

\par

We will often utilize the following estimates which follow from e.g. (3) in \cite[Chapter IV, Subsection 2.1]{GS2}. 
For any given $\alpha > 0$ there  exist positive constants  $C_1,C_2, A$, and $ B$ such that
\begin{equation}\label{GS2}
C_1 e^{A s^{1/\alpha}}\leq \sup_{n\in\mathbb N_0}\frac{s^n}{n!^\alpha}\leq C_2 e^{B s^{1/\alpha}}, \qquad
 s\geq 0.
\end{equation}

\begin{thm}\label{glavna} Let $\alpha > 0$.
\begin{itemize}
\item[i)] If   $f\in \mathcal S(\mathbb R_+^d)$ and $\{a_n(f)\}_{n\in\NN^d_0}\in \ell_{\alpha/2} (\NN_0^d)$, then $
\quad f \in \mathbf{G}^{\alpha}_{\alpha}(\RR^d_+).$
\item[ii)] If  $ f \in \mathbf{G}^{\alpha}_{\alpha}(\RR^d_+)$, then $\{a_n(f)\}_{n\in\NN^d_0}\in  \ell_{\alpha/2} (\NN_0^d)$.
\item[iii)]  If $a_n\in \ell_{\alpha/2}'(\NN_0 ^d)$ then  the series
$ \displaystyle \sum_{n\in\NN^d_{0}}a_n l_n$ converges in the topology of
$(\mathbf{G}^{\alpha}_{\alpha})' (\RR^d_+)$,
thus defining an element $ f \in (\mathbf{G}^{\alpha}_{\alpha})' (\RR^d_+).$ Moreover, $ a_n = a_n (f)$, $ n \in \NN^d_{0}$.
\item[iv)]  Let $ f \in (\mathbf{G}^{\alpha}_{\alpha})'(\RR^d_+).$
Then for the sequence of its Fourier-Laguerre coefficients we have $a_n(f)
\in \ell'_{\alpha/2}(\NN^d_{0}).$
\end{itemize}
\end{thm}

\begin{proof} We  prove  i) and ii). Parts iii) and iv) then follow by the usual duality arguments.

i) Let $f\in \mathcal S(\mathbb R_+^d)$ and $h>0$ be such that $\{a_n(f)\}_{n\in\NN^d_0}\in \ell^{\infty}_{[\vartheta_{h,\alpha/2}]}$. By the norm equivalence (cf. Remark \ref{rem:norme})
there exists $h'>0$ such that $\|\{a_n\}_{n\in\NN^d_0}\|_{\ell^{2}_{[\vartheta_{h',\alpha/2}]}} <\infty$.

We need to prove that
$$\eta_{h_1}^\alpha(f)=\sup_{N\in\NN_0}
\frac{\Vert E^N f\Vert_{L^2(\RR^d_+)}}{h_1^N N!^\alpha}<\infty, $$
for some  $h_1>0$.
Choose $h_1>0$ such that $\ds h'=B /h_1^{1/\alpha}$, where $B$  is the constant  in (\ref{GS2}).
Then, by employing the identity \eqref{Eigen} we get the estimate


 \begin{equation}\label{eq:inequalities01}
\frac{\Vert E^N f\Vert_{L^2(\RR^d_+)}}{h_1^N N!^\alpha}
=
\frac{\Big(\sum_{n\in\NN^d_0}|a_n|^2\cdot|n|^{2N}\Big)^{1/2}}{h_1^N N!^\alpha}
= \Big( \sum_{n\in\NN^d_0}\frac{|a_n|^2\cdot|n|^{2N}}{h_1^{2N} N!^{2\alpha}}\Big)^{1/2}\leq C_2 \|\{a_n\}_{n\in\NN^d_0}\|_{\ell^{2}_{[\vartheta_{h',\alpha/2}]}},
 \end{equation}
for any $N\in\mathbb N_0$. This shows that
%
$$\eta_{h_1}^\alpha(f)=\sup_{N\in\NN}\frac{\Big(\sum_{n\in\NN^d_0}|a_n|^2|n|^{2N}\Big)^{1/2}}{h_1^N N!^\alpha} \leq C_2 \|\{a_n\}_{n\in\NN^d_0}\|_{\ell^{2}_{[\vartheta_{h', \alpha/2}]}}  < \infty,
$$
i.e. $f\in \mathbf{G}^{\alpha}_{\alpha}(\RR^d_+)$.

\par

ii)
 Now suppose that $f\in \mathbf{G}^{\alpha}_{\alpha}(\RR^d_+)$, i.e.
 $\eta_h(f)<\infty$  for some $h>0$.
We need to show that there exists $h_1>0$ such that $\{a_n(f)\}_{n\in\NN^d_0}\in  \ell^{\infty}_{[\vartheta_{h_1, \alpha/2}]}$.
An application of \eqref{GS2} leads to
$$\ds e^{h_2|n|^{1/\alpha}}=e^{A  \cdot \big((\frac{h_2}{A})^{\alpha}|n|\big)^{1/\alpha}}
\leq
\frac{1}{C_1} \sup_{N\in \mathbb N} \frac{ |n|^N  }{h^N N!^{\alpha}}, $$
where $h_2 : = A/ h^{1/\alpha}$. Next we estimate

$$
\sup_{n\in\NN^d_0}|a_n| e^{h_2|n|^{1/\alpha}}
 \leq
\frac{1}{C_1}  \sup_{n\in\NN^d_0} \Big( |a_n|\cdot
\big(\sup_{N\in\NN}\frac{ |n|^N}{h^N N!^\alpha}\big)\Big)
$$
$$
 \leq  \frac{1}{C_1}
\sup_{N\in\NN}\frac{\Big(\sum_{n\in\NN^d_0}|a_n|^2|n|^{2N}\Big)^{\frac{1}{2}}}{N!^{\alpha}
h^{N}} \leq C \frac{1}{C_1}  \eta_h^\alpha(f),
$$
and conclude
$$ \Big(\sum_{n\in\NN^d_0} |a_n|^2 e^{2 h_1 n^{1/\alpha}}\Big)^{1/2}\leq
\Big(\sup_{n\in\NN^d_0} |a_n| e^{3h_1 n^{1/\alpha}}\Big)\cdot \Big(\sum_{n\in\NN^d_0} e^{-h_1 n^{1/2}}\Big)^{1/2}\leq  C\eta_h^\alpha(f) < \infty.$$
Hence, $\displaystyle \{a_n(f)\}_{n\in\NN^d_0}\in  \ell^{\infty}_{[\vartheta_{h_1, \alpha/2}]}$, and the proof is completed.

\end{proof}

\begin{rem}
We note that by Theorem \ref{glavna} it follows that for any sequence $\{a_n \}_{n\in\NN^d_{0}}
\in \ell^{\infty}_{[\vartheta_{h,\alpha/2}]} $
there exist $h_1>0$ and $C_1>0$ such that
$$
\|\{a_n \}_{n\in\NN^d_{0}} \|_{\ell^{\infty}_{[\vartheta_{h,\alpha/2}]}}
\leq C_1\cdot \eta_{h_1}^\alpha\big(  T( \{a_n \}_{n\in\NN^d_{0}} )\big),
$$
and conversely, for every $f\in \mathbf G^{\alpha,h}_{\alpha,h}(\mathbb R^d)$ there exist $h_1>0$ and $C_1 >0$ such that
$\{a_n \}_{n\in\NN^d_{0}} = T^{-1}(f)\in \ell^{\infty}_{[\vartheta_{h_1,\alpha/2}]}(\NN_0^d) $,
and
$$
\eta_{h}^\alpha\big( f)\leq C_1\cdot \|\{a_n \}_{n\in\NN^d_{0}} \|_{      \ell^{\infty}_{[\vartheta_{h_1,\alpha/2}]}  }.
$$
Therefore,   the mapping $T:{\ell_{\alpha/2}(\NN_0^d)}\rightarrow \mathbf G^{\alpha}_{\alpha}(\RR^d_{+})$ is a well defined topological isomorphism.

Thus, Theorem \ref{glavna} implies that
 $\mathbf{G}^{\alpha}_{\alpha}(\RR^d_+)=\mathcal{G}_{\alpha}(\RR^d_+)$,
 $\alpha > 0 $.
\end{rem}

When considering the Beurling space $ \mathbf{g}^{\alpha}_{\alpha}(\RR^d_+)$, $ \alpha >0$,
we have the following analogue of Theorem \ref{glavna}. The proof is  similar to the proof of Theorem \ref{glavna},
and is left to the reader.

\begin{thm}\label{glavnabeurling} Let $\alpha > 0$.
\begin{itemize}
\item[i)] If $f\in \mathcal S(\mathbb R_+^d)$ and $\{a_n(f)\}_{n\in\NN^d_0}\in \ell_{0,\alpha/2}(\NN^d_0)$, then
$ f \in \mathbf{g}^{\alpha}_{\alpha}(\RR^d_+).$
\item[ii)] If  $ f \in \mathbf{g}^{\alpha}_{\alpha}(\RR^d_+)$, then
$a_n(f)\in  \ell_{0,\alpha/2}(\NN^d_0) $.
\item[iii)]  If $a_n\in \ell_{0,\alpha/2}'(\NN^d _0)$ then the series
$$\sum_{n\in\NN^d_{0}}a_n l_n$$ converges in the topology of dual space $(\mathbf{g}^{\alpha}_{\alpha})'(\RR^d_+)$,
thus defining an element $ f \in (\mathbf{g}^{\alpha}_{\alpha})'(\RR^d_+).$
Moreover, $ a_n = a_n (f)$, $ n \in \NN^d_{0}$.
\item[iv)]  Let $ f \in (\mathbf{g}^{\alpha}_{\alpha})'(\RR^d_+).$ Then for the sequence
of its Fourier-Laguerre coefficients we have $a_n(f)
\in \ell'_{0,\alpha/2}(\NN^d_{0}).$
\end{itemize}
\end{thm}

Thus,  $\mathbf{g}^{\alpha}_{\alpha}(\RR^d_+)=\mathcal{G}_{0,\alpha}(\RR^d_+)$, and the mapping
 $T:{\ell_{0, \alpha/2}(\NN_0^d)}\rightarrow \mathbf g^{\alpha}_{\alpha}(\RR^d_{+})$ is a well defined topological isomorphism. 
 
Moreover, if we combine Theorem \ref{glavna} and Theorem \ref{D Th3.6}, we obtain
\begin{equation}
\mathbf G^\alpha _\alpha(\RR^d_+) = G_\alpha^\alpha(\RR^d_+),
\quad
\alpha\geq 1, \quad  \text{and} \quad
\mathbf g ^\alpha _\alpha (\RR^d_+)= g_\alpha^\alpha(\RR^d_+), \quad
 \alpha>1.
    \end{equation}

Moreover,
\begin{equation}
\mathbf{G}^{\alpha}_{\alpha}(\RR^d_+)\neq G^{\alpha}_{\alpha}(\RR^d_+),  \quad 0<\alpha<1, \quad  \text{and} \quad  \mathbf{g}^{\alpha}_{\alpha}(\RR^d_+)\neq g^{\alpha}_{\alpha}(\RR^d_+),
 \quad 0<\alpha\leq 1.
    \end{equation}

By Theorem  \ref{glavna} and Theorem \ref{glavnabeurling} it immediately follows that   $(\mathbf{G}^{\alpha}_{\alpha})'(\RR^d_+)$ and $(\mathbf{g}^{\alpha}_{\alpha})'(\RR^d_+)$, are topologically isomorphic to sequence spaces  $\ell_{\alpha/2}'(\NN^d_0)$ and $\ell_{0,\alpha/2}'(\NN^d_0)$, $\alpha>0$, respectively.

\begin{rem}
Let now  $\alpha=0$, and $h>0$. We define $ \mathbf G^{0,h}_{0,h}(\RR^d_+)$, as
the set of all smooth functions $ f $ such that
$\| E^N f\|_{L^2(\RR^d_+) }\leq C h^{|N|}$, for some $C>0$.
Then
$$\ds \mathbf{G}^{0,h}_{0,h}(\RR^d_+)
= \left \{\sum_{n\in \NN^d_0} a_nl_n \quad | \quad  a_n=0 \;\;  \mbox{if}\;\;\ |n | >h \right \}, $$
and it follows that
$
\mathbf{G}^{0}_{0}(\RR^d_+) = \bigcup_{h>0} \mathbf{G}^{0,h}_{0,h}(\RR^d_+)
$
is the set of all (finite) linear combinations of Laguerre functions.
Therefore $\mathbf{G}^{0}_{0}(\RR^d_+)=\mathcal{G}_{0}(\RR^d_+)$,
see subsection \ref{subsec:formal}.  In addition,
$
\mathbf{g}^{0} _{0}(\RR^d_+)=  \bigcap_{h>0} \mathbf{G}^{0,h}_{0,h}(\RR^d_+) = \{0\}.
$
\end{rem}

We end this section with the kernel theorem for  $P-$spaces on positive orthants
which extends the corresponding result for $G$-type spaces.

\begin{thm}
    Let $\alpha>0$. We have the following:
$$\mathbf{G}_{\alpha}^{\alpha}(\RR^{d_1}_{+})\hat{\otimes}\mathbf{G}_{\alpha}^{\alpha}(\RR^{d_2}_{+})\cong \mathbf{G}_{\alpha}^{\alpha}(\RR^{d_1+d_2}_{+})\quad \mbox{and}\quad (\mathbf{G}_{\alpha}^{\alpha})'(\RR^{d_1}_{+})\hat{\otimes}(\mathbf{G}_{\alpha}^{\alpha})'(\RR^{d_2}_{+})\cong (\mathbf{G}_{\alpha}^{\alpha})'(\RR^{d_1+d_2}_{+}),$$

$$\mathbf{g}_{\alpha}^{\alpha}(\RR^{d_1}_{+})\hat{\otimes}\mathbf{g}_{\alpha}^{\alpha}(\RR^{d_2}_{+})\cong \mathbf{g}_{\alpha}^{\alpha}(\RR^{d_1+d_2}_{+})\quad \mbox{and}\quad (\mathbf{g}_{\alpha}^{\alpha})'(\RR^{d_1}_{+})\hat{\otimes}(\mathbf{g}_{\alpha}^{\alpha})'(\RR^{d_2}_{+})\cong (\mathbf{g}_{\alpha}^{\alpha})'(\RR^{d_1+d_2}_{+}).$$
\end{thm}

\begin{proof}
   Recall, the sequence spaces $\ell_{\alpha/2}(\NN_0^d)$ and $\ell_{0,\alpha/2}(\NN^d_0)$ are nuclear. The claim follows by calculations similar to the ones given in the proof of 
   \cite[Theorem 6.4]{SSB}, and  \cite[Theorem 4.4]{SSBb}. We leave details to the reader.
\end{proof}

\section{Alternative norms for defining $P-$spaces} \label{sec4}

In Definition \ref{def:PilipovicOrthant} we introduced $P-$spaces by considering
$L^2-$ norms of the iterates of the Laguerre operator, and in Theorem \ref{glavna} we relate them with spaces of sequences of the corresponding Laguerre coefficients.
For such spaces of sequences we noted in  Remark \ref{rem:norme}
that the weighted $\ell^2-$norm,  $\|\cdot\|_{\ell^{2}_{[\vartheta_{h,\alpha}]}}$, which gives rise to $\ell_{\alpha}(\NN^d_0)$ and $\ell_{0,\alpha}(\NN^d_0)$,
can be replaced by a weighted $\ell_p$ norm, $p\in [1,\infty]$,
$\|\cdot\|_{\ell^{p}_{[\vartheta_{h,\alpha}]}}$,
without changing the spaces $\ell_{\alpha}(\NN^d_0)$ and $\ell_{0,\alpha}(\NN^d_0)$, and their topological structure.

In this subsection we demonstrate that the same is true for $P-$spaces. More precisely, the $L^2-$ norm, $\|\cdot\|_{L^2(\RR^d_{+})}$, in (\ref{eta})
can be replaced by $\|\cdot\|_{L^p(\RR^d_{+})}$,  $p\in [1,\infty]$, yielding
the same $P-$space, and preserving its topological structure.

\par

As a first step, we estimate $L^p-$norms  of the  Laguerre functions $l_n$, $ n \in \NN_0$.
Although this result is probably known, we give its proof for the sake of completeness.

\begin{lem} \label{lem:duran} Let  $p \in [1,\infty]$.
For any Laguerre function  $l_n$, $n\in \NN_0^d$, the following  estimate  holds:
\begin{equation}\label{lager}
\ds \|l_n\|_{L^p(\RR^d_{+})}\leq  C_1 \prod_{j=1}^d n_j^2\leq C_1 |n|^{2d},
\end{equation}
where the constant $C_1 >0$ is independent  on $n\in \NN_0^d$ and $p \in [1,\infty]$.
\end{lem}

\begin{proof}
Let $p\in [1,\infty)$ (the case $p=\infty$ is obvious), and
let us consider the one dimensional case, i.e.  $n\in \NN_0$.
By \cite[Thm. 3]{DLP} it follows that
\begin{equation} \label{lm:laguerre-estimate}
    |x^k l_n(x)|\leq 4^k (n+1)(n+2)\cdots (n+k), \qquad x\in\RR_{+},
\end{equation}
wherefrom
$$
\int_0^{\infty} |l_n(x)|^p dx = \left ( \int_0^{1} +\int_1^{\infty} \right )
|l_n(x)|^p dx
\leq 1 + 16^p (n+1)^p(n+2)^p\int_{1}^{\infty}\frac{1}{x^{2p}}dx   \leq C_1 ^p n^{2p},
$$
for a suitable $C_1>0$  that does not depend on $n \in \NN_0$ and $p\in [1,\infty)$.
The same calculation holds for $n\in\NN^d_{0}$, and we get \eqref{lager}.
\end{proof}

\begin{prop} \label{equiv-norm1}
Let $p\in [1,\infty]$, $\alpha>0$, and     $f\in C^{\infty}(\RR^d_{+})$. There
exist $h>0$ and $C>0$ such that
\begin{equation}\label{pnorme}
\|E^{N} f\|_{L^p(\RR^d)} \leq   C h^N N!^{\alpha},
\end{equation}
for every $N \in \NN_{0}$,
if and only if
there exists $h_1>0$ such that
$f\in \mathbf G^{\alpha,h_1}_{\alpha,h_1}(\RR^d_+)$.
\end{prop}

\begin{proof}
Suppose $p\in [1,\infty)$, without loss of generality.

$(\Rightarrow)$ By the H\"{o}lder inequality, (\ref{pnorme}) and Lemma \ref{lem:duran},
for any given $n, N \in \NN_{0}$, we have
 $$|n|^N |a_n(f)|= |\int_{\RR^d_{+}} E^N(f(x)) l_n (x)dx| \leq  \|E^N f \|_{L^p(\RR^d_{+})}\cdot \|l_n\|_{L^{p'} (\RR^d_{+})} \leq   C_1 h^N N!^{\alpha} |n|^{2d},
 $$
for a suitable constant $C_1>0$, where  $\ds \frac{1}{p}+\frac{1}{p'}=1$.
Hence, there exist $C_2, C_3, h_3 >0$  such that
$$|a_n(f)|\leq \sup_{N\in\NN_0} \frac{ C_2 h^N N!^{\alpha}} {|n|^N}\cdot |n|^{2d}\leq C_3 e^{-h_3|n|^{1/\alpha}}\cdot |n|^{2d},$$
and therefore
$
\ds \|\{a_n(f) \}_{n\in\NN^d_{0}}\|_{\ell^{\infty}_{[\vartheta_{h_3/2,\alpha/2}]}}<\infty.
$
By Remark \ref{rem:norme} there exist  $h_4>0$, so that
$\ds \|\{a_n(f) \}_{n\in\NN^d_{0}} \|_{\ell^{2}_{[\vartheta_{h_4,\alpha/2}]}(\NN_0^d)}<\infty, $ and the claim follows from Theorem \ref{glavna}.

$(\Leftarrow)$
Let $f\in \mathbf G^{\alpha, h_1}_{\alpha,h_1}(\RR_{+}^d)$  for some $h_1>0$, and let
$\{a_n(f) \}_{n\in\NN^d_{0}}$ be the sequence of its Laguerre coefficients. By Theorem \ref{glavna} we have that $|a_n(f)|\leq C_2 e^{-h_2|n|^{1/\alpha}}$ for some
$h_2,C_2>0$. This and Lemma \ref{lem:duran} imply
$ \ds
\|\sum_{n\in\NN^d_{0}}a_n(f)\cdot  l_n \|_{L^p(\RR^d_{+})}
<\infty.
$
Moreover, from the right hand side inequality in  (\ref{GS2}) we deduce that
$$
|n|^N e^{-\frac{h_2}{2}|n|^{1/\alpha}} \leq  C_3 h^N N!^{\alpha},
$$
for some positive constants $ h$ and $  C_3 $, wherefrom
$$\|E^{N} f\|_{L^p(\RR^d _+)}\leq
\sum_{n\in\NN^d_{0}}|n|^N |a_n| \cdot \|l_n\|_{L^p(\RR^d_{+})} \leq
C_1 C_2 \sum_{n\in\NN^d_{0}} |n|^N e^{-\frac{h_2}{2}|n|^{1/\alpha}} |n|^{2d} e^{-\frac{h_2}{2}|n|^{1/\alpha}}    $$
$$\leq
C_1 C_2 C_3  h^N N!^{\alpha} \sum_{n\in\NN^d_{0}}  |n|^{2d} e^{-\frac{h_2}{2}|n|^{1/\alpha}} = C  h^N N!^{\alpha}, $$
with $ C = C_1 C_2 C_3$, which concludes the proof.
\end{proof}

\section{Relation between Pilipovi\'{c} space on $\RR^d$ and $P-$spaces on $\RR^d_+$}
\label{sec5}

In this section we extend some results from \cite{SSSB} where
the relation between $G-$type spaces and Gelfand-Shilov spaces is studied. 
Our goal here is to find natural global counterparts of $P-$spaces on positive orthants.

By  $\mathcal{H}_{\alpha, even}(\RR^d)$
we denote the space  of ``even'' functions in $\mathcal{H}_{\alpha}(\RR^d)$, i.e.
$$
f \in  \mathcal{H}^{\alpha}_{\alpha, even}(\RR^d) \quad \Longleftrightarrow
\quad
f(x) = f (x_1, x_2,\dots, x_d) = f (x_1,\dots, x_{j-1},-x_j,x_{j+1},\dots x_d),
$$
for every $ x = (x_1, x_2,\dots, x_d) \in\RR^d$.
In a similar fashion, one can introduce  $\mathcal S^{\alpha}_{\alpha, even}(\RR ^d)$,
$\Sigma ^{\alpha}_{\alpha, even}(\RR^d)$, and  $\mathcal H_{0,\alpha, even}(\RR ^d)$, $ \alpha > 0$. We start with an easy observation.

\begin{lem} \label{pprk1} Let $\alpha > 0$.
The space $\mathcal H_{\alpha, even}(\RR ^d)$
is a closed subspace of $\mathcal H_{\alpha}(\RR ^d)$.
It consists of $f\in \mathcal H_{\alpha}(\RR^d)$ which can be represented as $f=\sum_{n\in\NN^d_0} a_{2n} (f) h_{2n}$ where $\{a_{2n} (f) \}_{n\in\NN^d_0}\in \ell_{\alpha} (\NN^d_0)$.
\end{lem}

\begin{proof}
    The proof for $ \alpha \geq 1/2$ is given in \cite{SSSB}.
    We notice that a more general case $ \alpha >0$  follows by using the same arguments as in the proof of \cite[Proposition 2.2]{SSSB}.
\end{proof}

\begin{rem} The meaning of $\{a_{2n}\}_{n\in\NN^d_0}\in \ell_{\alpha} (\NN^d_0) $
should be understood in the following way: $\{a_{2n}\}_{n\in\NN^d_0}\in \ell_{\alpha} (\NN^d_0) $ if and only if the sequence $\{b_k\}_{k\in\NN^d_0}$ belongs to
$\ell_{\alpha} (\NN^d_0) $ where $ b_k = a_k $ for all
indices $k=2n$, $n\in\NN^d_0$,  and all other elements of $\{b_k\}_{k\in\NN^d_0}$ are equal to $0$.
In the sequel, whenever we use this notation
(also when considering sequences in $\ell_{0,\alpha} (\NN^d_0) $), it will have exactly this meaning.
\end{rem}

To establish the relationship between
$\mathcal H_{\alpha, even}(\RR ^d)$
and $\mathbf{G}^{\alpha}_{\alpha}(\RR_{+}^d)$, $ \alpha > 0$,
we first introduce some notation and recall  two important results from  \cite{SSSB}.
Let 
$$
v:\RR^d\rightarrow \overline{\RR_{+}^d}, \;\;\; v(x)=(x_1^2,\dots, x_d^2),
\quad \text{and} \quad
w:\overline{\RR_{+}^d} \rightarrow \overline{\RR_{+}^d}, \;\;\; w(x)=(\sqrt{x_1},\dots \sqrt{x_d}),$$
$\mathbf{1/2}=(1/2,\ldots, 1/2)\in\RR^d_+$, and $\mathbf{3/2}=(3/2,\ldots, 3/2)\in\RR^d_+$.

\begin{prop}  \cite[Proposition 3.1]{SSSB}
\label{Prop3.1}
If $\ds f=\sum_{n\in\NN^d_0}{a_n(f)l_n}\in G_{\alpha}^{\alpha}(\RR^d_{+})$,
$\alpha\geq 1$. Then $f\circ v\in \mathcal{S}^{\alpha/2}_{\alpha/2, even}(\RR^d)$ and we have
$ \displaystyle f\circ v=\sum_{n\in\NN^d_0} b_{2n}h_{2n},$
where $\{b_{2n}\}_{n\in\NN^d_0}\in \ell_{\alpha/2}(\NN^d_0)$, is given by
\begin{equation}\label{luh}
b_{2n}=\frac{(-1)^n  \pi^{d/4} \sqrt{(2n)!} }{2^{|n|}n!}\cdot \sum_{k\in\NN_0^d} a_{k+n}(f){k-{\bf 1/2}\choose k}, \quad n\in\NN_0^d.\end{equation}

Moreover, the mapping $f \mapsto f\circ v $ is a continuous injection from
$ G_{\alpha}^{\alpha}(\RR^d_{+})$ to $\mathcal{S}^{\alpha/2}_{\alpha/2, even}(\RR^d)$.
\end{prop}

\begin{prop}  \cite[Proposition 3.2]{SSSB}
\label{Prop3.2}
Let $\alpha\geq 1$, and $\ds f=\sum_{n\in\NN_0^d} a_{2n}(f) h_{2n}\in \mathcal{S}^{\alpha/2}_{\alpha/2, even}(\RR^d) $. Then
$\ds f|_{\RR^d_{+}} \circ w$, belongs to $ G_{\alpha}^{\alpha}(\RR^d_{+})$, and 
$\ds f|_{\RR^d_{+}} \circ w = \sum_{n\in\NN_0^d} b_n l_n$, where
$\{b_n \}_{n\in\NN_0^d}\in \ell_{\alpha/2}(\NN^d_0)$, and $b_n $ is given by
\begin{equation}\label{hul}
b_n=\frac{(-1)^n 2^{|n|}}{\pi^{d/4}}\sum_{k\in\NN_0^d}{k-{ \bf 3/2} \choose k}\frac{(-1)^{|k|} 2^{|k|}  (k+n)! a_{2(k+n)}  (f) }{\sqrt{(2(k+n))!}}, \quad n\in \NN_0^d.
\end{equation}

Moreover, the mapping $f \mapsto f\circ v $ is a continuous injection from
$\mathcal{S}^{\alpha/2}_{\alpha/2, even}(\RR^d)$ to $ G_{\alpha}^{\alpha}(\RR^d_{+})$.
\end{prop}

%

By a careful examination of the proof of \cite[Proposition 3.1]{SSSB} we conclude that
almost all estimates given there for  $ \alpha \geq 1$,
are also true when $ \alpha \in (0,1)$, except the inequality
\begin{equation} \label{inequality (3.4)}
 |n|^{1/\alpha}+|m|^{1/\alpha}\leq 2(|n|+|m|)^{1/\alpha},\quad m,n\in \NN^d_0
\end{equation}
(see (3.4) in \cite{SSSB}).
To prove extensions of Proposition \ref{Prop3.1} and Proposition \ref{Prop3.2}
we use the following estimates instead.
Let $s=(s_1,s_2,\dots, s_d)\in\RR_{+}^d$ and
$\|s \|_p = (\sum_{k=1} ^d |s_k |^p)^{1/p}$, $s \in \RR_{+}^d$, $p \in \mathbb N.$
Then there exist $M_1,M_2>0$ (depending on $\alpha>0$, and $p$)
such that $M_1\|s\|_p \leq  (\sum_{j=1}^d s_j^{\frac{1}{\alpha}}) ^{\alpha}
\leq M_2 \|s\|_p,$
which implies
\begin{equation} \label{eq:modifikacija}
|m|^{1/\alpha}+|n|^{1/\alpha}\leq
C|m+n|^{1/\alpha}, \qquad m,n \in  \NN^d _0, \quad \alpha >0.
\end{equation}

Thus, if we replace \eqref{inequality (3.4)}
by \eqref{eq:modifikacija}, and otherwise follow the proof of  \cite[Proposition 3.1]{SSSB}, we get the following theorem for $P-$spaces.

\begin{thm} \label{even-spaces-1}
Let $\alpha>0$.
If $\ds f=\sum_{n\in\NN^d_0}{a_n(f)l_n}\in \mathbf{G}_{\alpha}^{\alpha}(\RR^d_{+})$, then $f\circ v\in \mathcal{H}_{\alpha/2, even}(\RR^d)$ and we have
$$f\circ v=\sum_{n\in\NN^d_0} b_{2n}h_{2n},$$
where $\{b_{2n}\}_{n\in\NN^d_0}\in \ell_{\alpha/2}(\NN^d_0)$  and  (\ref{luh}) holds. If, on the other hand,
$$
f=\sum_{n\in\NN_0^d} a_{2n}(f) h_{2n}\in \mathcal{H}_{\alpha/2, even}(\RR^d),
$$
then $ \ds f|_{\RR^d_{+}}\circ w  \in  \mathbf{G}_{\alpha}^{\alpha}(\RR^d_{+})$, 
and we have $\ds f|_{\RR^d_{+}} \circ w = \sum_{n\in\NN_0^d} b_n l_n$, with
$\{b_n \}_{n\in\NN_0^d}\in \ell_{\alpha/2}(\NN_0^d)$,
and such that (\ref{hul}) holds.
\end{thm}

%

\begin{cor}  Let $\alpha>0$. The mapping
 $f\mapsto f\circ v, $ from $\mathbf{G}_{\alpha}^{\alpha}(\RR^d_{+})$ to $\mathcal{H}_{\alpha/2, even}(\RR^d) $ is a topological isomorphism, and its inverse mapping from 
$\mathcal{H}_{\alpha/2, even}(\RR^d) $ to 
$  \mathbf{G}_{\alpha}^{\alpha}(\RR^d_{+})$ is given by
$\ds \psi\mapsto\psi|_{\RR^d_{+}}\circ w$.
\end{cor}

\begin{proof}
Let $\alpha>0$. The mapping  $f\mapsto f\circ v$ given in Theorem \ref{even-spaces-1}
is a bijection from $ \mathbf{G}_{\alpha}^{\alpha}(\RR^d_{+})$ to 
$ \mathcal{H}_{\alpha/2, even}(\RR^d) $, and its inverse is given by
$\ds (f\circ v)|_{\RR^d_{+}}\circ w$.   
Let $B$ be a bounded set in $\mathbf{G}_{\alpha}^{\alpha}(\RR^d_{+})$, and let
$ f=\sum_{n\in\NN^d_0}{a_n(f)l_n} \in B$. Then by 
Theorem \ref{glavna} it follows that there exists $h>0$ which depends only on $B$
such that 
\begin{equation} \label{bddsequences}
\{a_n(f)\}_{n\in\mathbb N_0^d} \in \ell^{\infty}_{[\vartheta_{h,\alpha}]}(\NN_0^d)\;\;
\mbox{ i.e.}\;\;  |a_n(f)| e^{h |n|^{1/\alpha}} < \infty.    
\end{equation}
It follows that sequences 
$\{b_{2n}\}_{n\in\NN_0^d}$ given by  \eqref{luh}
for all $\{a_n(f)\}_{n\in\mathbb N_0^d}$ such that  \eqref{bddsequences}
holds, satisfy 
$$
|b_{2n}| e^{h_1 |2n|^{1/2\alpha}} < \infty \;\;\; \text{ for some } \;\;\;
h_1>0 \;\; \text{ and all} \;\;\; n\in\NN_0^d. 
$$
This means that the set 
$$
\big \{  f|_{\RR^d_{+}} \circ w = \sum_{n\in\NN_0^d} b_n l_n \; | \; f \in B \big \}
$$
is bounded in  $\mathcal{H}_{\alpha/2, even}(\RR^d)$. 
In other words, the image of a bounded set in $ \mathbf{G}_{\alpha}^{\alpha} (\RR^d_{+})$ 
under the mapping $f\mapsto f\circ v$ is a bounded set in $\mathcal{H}_{\alpha/2, even}(\RR^d)$. 

Since $\mathbf{G}_{\alpha}^{\alpha}(\RR^d_{+})$ is a bornological space (by Theorem \ref{glavna} and the fact that $ \ell_{\alpha/2}(\NN_0^d)$ is a
$(DFN)-$space), it follows that the mapping $f\mapsto f\circ v$ is continuous. 

\par 

By \eqref{eq:mappingTHermite} it follows that 
$\mathcal{H}_{\alpha/2}(\RR^d)$ is a $(DFS)$ space, and the same is true for 
$  \mathcal{H}_{\alpha/2, even}(\RR^d) $.
Then, reasoning in the same manner as here above, and by using (\ref{hul}),
we conclude that the mapping $\ds \psi \mapsto  \psi|_{\RR^d_{+}}\circ w$ 
maps  bounded sets in $  \mathcal{H}_{\alpha/2, even}(\RR^d) $ onto bounded sets in
$ \mathbf{G}_{\alpha}^{\alpha}(\RR^d_{+})$, and therefore it is continuous.   

Thus $\mathbf{G}_{\alpha}^{\alpha}(\RR^d_{+})$ and $\mathcal{H}_{\alpha/2, even}(\RR^d) $ are topologically isomorphic.
\end{proof}

By invoking Proposition \ref{Prop3.2} and its proof we show that
a result  analogous to Theorem \ref{even-spaces-1} holds for the Beurling case as well.

\begin{thm} \label{even-spaces-2} Let $\alpha>0$.
If $\ds f=\sum_{n\in\NN^d_0}{a_n(f)l_n}\in \mathbf{g}_{\alpha}^{\alpha}(\RR^d_{+})$, then $f\circ v\in \mathcal{H}_{0,\alpha/2, even}(\RR^d)$ and 
$f\circ v=\sum_{n\in\NN^d_0} b_{2n}h_{2n}$,
where $\{b_{2n}\}_{n\in\NN^d_0}\in \ell_{0,\alpha/2}(\NN^d_0)$ is such that (\ref{luh}) holds. If, on the other hand,
$$
f=\sum_{n\in\NN_0^d} a_{2n}(f) h_{2n}\in \mathcal{H}_{0,\alpha/2, even}(\RR^d),
$$
then  $\ds f|_{\RR^d_{+}} \circ w\in  \mathbf{g}_{\alpha}^{\alpha}(\RR^d_{+})$. Moreover, 
$\ds f|_{\RR^d_{+}} \circ w = \sum_{n\in\NN^d_0} b_n l_n$ with
$\{b_n \}_{n\in\NN_0^d}\in \ell_{0,\alpha/2}(\NN_0^d)$,
and  (\ref{hul}) holds.
\end{thm}

By using the fact that  $\mathbf{g}_{\alpha}^{\alpha}(\RR^d_{+})$ and  $\mathcal{H}_{0,\alpha/2, even}(\RR^d)$
are bornological spaces (since they are  metrizable being  Frechet spaces, cf. \cite[Prop. 24.13, pp. 283]{MV}), we have the following corrollary of Theorem \ref{even-spaces-2}.

\begin{cor}
 Let $\alpha>0$. The mapping $f\mapsto f\circ v$ from
$\mathbf{g}_{\alpha}^{\alpha}(\RR^d_{+})$ to $ \mathcal{H}_{0,\alpha/2, even}(\RR^d)$
is a topological isomorphism, and its inverse mapping is given by
$\ds \psi\mapsto\psi|_{\RR^d_{+}}\circ w$.
\end{cor}

\section{Flat Pilipovi\'c spaces, extension to $ \RR_{\flat}$} \label{sec6}

It is observed by Toft in \cite{Toft1} that together with Pilipovi\'c spaces
$\mathcal{H}_{\alpha/2}(\RR^d)$ and $\mathcal{H}_{0,\alpha/2}(\RR^d)$
which are essentially linked to the weights of the form
$\vartheta_{h,\alpha}(\cdot) = e^{h|\cdot|^{1/(2\alpha)}}, $
$ \alpha, h> 0$,
one can observe another type of weights, which gives rise to "fine tuning" between Pilipovi\'{c} spaces when $ \alpha \in (0, 1/2)$.

Such spaces are called flat Pilipovi\'c spaces in  \cite{Toft1}. In this section we study their counterparts on positive orthants.

Following \cite[Section 3.1]{Toft1},
we consider the extension of $\RR_{+}$ given by
$$\RR_{\flat}:=\RR_{+}\cup \{\flat_{\sigma}:\sigma>0\}, \quad \text{and } \quad
\overline{\RR_{\flat}} = \RR_{\flat} \cup \{0\},
$$
where the usual ordering holds in $\RR_{+}$, and new elements are ordered in the following manner:
$$
x_1<\flat_{\sigma_1}< \flat_{\sigma_2}<x_2  \quad \text{if}
\quad
\sigma_1<\sigma_2 \quad \text{and} \quad x_1< 1/2 \leq x_2,
\quad x_1, x_2 \in \RR_{+}.
$$
In addition, $ \flat_\infty := \frac{1}{2} $.

Let $\flat_{\sigma}\in \RR_{\flat}\setminus \RR_{+}$, $ \sigma >0$,  and $h>0 $. We define $\vartheta_{h,\flat_{\sigma}}(n) : = h^{|n|}\cdot n !^{1/2\sigma}$,
$n \in \NN_0^d$,
and consider the space   $\ell^{\infty}_{[\vartheta_{h,\flat_{\sigma}}]}$
 of  sequences  $\{a_n\}_{n\in\NN^d_0}$ satisfying
\begin{equation}\label{koef.flat}
\ds \|\{a_n\}_{n\in\NN^d_0}\|_{\ell^{\infty}_{[\vartheta_{h,\flat_{\sigma}}]}(\NN_0^d)}=\|\{a_n\cdot \vartheta_{h,\flat_{\sigma}}(n) \}_{n\in\NN^d_0}\|_{l^{\infty}}=
\sup_{n\in\NN^d_0}\big| a_n \big| \cdot  h^{|n|} \cdot n!^{\frac{1}{2\sigma}} <\infty.
\end{equation}
Moreover,
$$
\ds \ell_{\flat_{\sigma}}(\NN^d_0):=\bigcup_{h>0} \ell^{\infty}_{[\vartheta_{h,\flat_{\sigma}}]}(\NN_0^d)
\quad \text{and} \quad \ell_{0,\flat_{\sigma}}(\NN_0^d) :=
\ds \bigcap_{h>0} \ell^{\infty}_{[\vartheta_{h,\flat_{\sigma}}]}(\NN_0^d),
$$
in the sense of inductive and projective limit topology, respectively.

We first show inclusion relations between these spaces.

\begin{lem}
Let $\alpha,\beta\in \RR_{\flat}$ and $\alpha<\beta$. Then $\ell_{\alpha}(\NN^d_0)\subseteq \ell_{\beta}(\NN^d_0)$ and $\ell_{0,\alpha}(\NN^d_0)\subseteq \ell_{0,\beta}(\NN^d_0)$.
\end{lem}
\begin{proof}
We will prove the Roumieu case $\ell_{\alpha}(\NN^d_0)\subseteq \ell_{\beta}(\NN^d_0)$, and leave the Beurling case for the reader. In addition, to simplify calculations, we will use $|n|!$ instead of $n!$ without losing generality.

Note that the claim is obvious when $\alpha, \beta\in \RR_{+} $ or $\alpha=\flat_{\sigma_1}, \beta=\flat_{\sigma_2}, \sigma_1<\sigma_2$.

Let $\alpha=\flat_{\sigma},\sigma>0$,
 $\beta=\frac{1}{2}$, and let $h>0$ be chosen so that$\{a_n\}_{n\in\NN^d_0}
 \in \ell^{\infty}_{[\vartheta_{h,\flat_{\sigma}}]}(\NN_0^d) $, i.e.
$\sup_{n\in\NN^d_0} |a_n|\cdot  h^{|n|}\cdot |n|!^{\frac{1}{2\sigma}}<\infty$. Using  the estimate $|n|^{|n|} \leq |n|! e^n$
we get
$$|a_n| e^{h_1|n|}=
\leq  |a_n| \cdot |n|!^{\frac{1}{2\sigma}}\cdot  e^{\frac{|n|}{2\sigma}}\cdot \frac{e^{h_1|n|}}{|n|^{\frac{|n|}{2\sigma}}}
\leq C  |a_n|\cdot |n|!^{\frac{1}{2\sigma}}\cdot h^{|n|}<\infty,
$$
for a suitable $C>0$,
where $h_1 > 0 $ is chosen so that
$ \displaystyle
e^{\frac{1}{2\sigma}}
e^{h_1}   \leq h |n|^{-\frac{1}{2\sigma}},
$
for $n \in \NN^d_0$ large enough.
It follows that $\{a_n\}_{n\in\NN^d_0} \in \ell_{1/2}(\NN^d_0)$.

Next we  prove the claim when $\alpha\in (0,\frac{1}{2})$ and $ \beta=\flat_{\sigma}$, $ \sigma>0$. Let $h>0$ be chosen so that such that $\ds |a_n|\cdot e^{h|n|^{1/(2\alpha)}}<\infty$. Then we have
$$
|a_n| h_1 ^{|n|} \cdot  |n|!^{\frac{1}{2\sigma}} \leq
|a_n| h_1 ^{|n|} e^{h|n|^{\frac{1}{2\alpha}}} \; \frac{|n|^{\frac{|n|}{2\sigma}}}{e^{h|n|^{1/(2\alpha)}}}
\leq |a_n|  e^{h |n|^{\frac{1}{2\alpha}}}
\left( \frac{h_1|n|^{\frac{1}{2\sigma}}}{ e^{h|n|^{\frac{1}{2\alpha} - 1}}} \right)^{|n|}, \quad h_1 > 0.
$$
Since $\ds 1> 2\alpha $, we have that  $\ds \sup_{n\in\NN_0^d}
|n|^{\frac{1}{2\sigma}} e^{-h|n|^{\frac{1}{2\alpha} - 1}}\leq C_h$, so we can   choose $h_1>0$ such that $C_h h_1 < 1$, wherefrom it follows that
$$|a_n| h_1 ^{|n|} |n|!^{\frac{1}{2\sigma}}\leq  |a_n|  e^{h |n|^{\frac{1}{2\alpha}}}<\infty, $$
i.e.
$\{a_n\}_{n\in\NN^d_0} \in \ell_{\flat_{\sigma}}(\NN^d_0)$,
$\sigma>0$.

The inclusion  $\ell_{0,\alpha}(\NN^d_0)\subseteq \ell_{0,\beta}(\NN^d_0)$ can be proved in a similar fashion.
\end{proof}

\begin{rem}
The sequence
$\{e^{-|n|}\}_{n\in\NN^d_0}$ satisfies
$$
\{e^{-|n|}\}_{n\in\NN^d_0} \in \ell_{1/2}(\NN^d_0)
\quad \text{ and} \quad \{e^{-|n|}\}_{n\in\NN^d_0} \not\in \bigcup_{\sigma>0} \ell_{\flat_{\sigma}}(\NN^d_0),
$$
hence the inclusion  $\ds\bigcup_{\sigma>0} \ell_{\flat_{\sigma}}(\NN^d_0)\subset \ell_{\frac{1}{2}}(\NN^d_0)$ is strict.
\end{rem}


The next result can be proved in a similar way as Proposition \ref{lm:nizovi-potapanje}.

\begin{prop}
Fix an arbitrary $\alpha=\flat_{\sigma}$ and $h_1>h_2$. Then the canonical inclusion
$ \ds \ell^{\infty}_{[\vartheta_{h_1,\flat_{\sigma}}]}(\NN_0^d)\rightarrow \ell^{\infty}_{[\vartheta_{h_2,\flat_{\sigma}}]}(\NN_0^d)$ is nuclear.
\end{prop}

We recall that in \cite{Toft1}, the spaces of  formal expansions
$$\mathcal H_{\flat_{\sigma}}(\RR^d)=\{\sum_{n\in\NN^d_0} a_n h_n: \{a_n\}_{n\in\NN^d_0}\in \ell_{\flat_{\sigma}}(\NN^d_0)\},
$$
$$ \mathcal H_{0,\flat_{\sigma}}(\RR^d)=\{\sum_{n\in\NN^d_0} a_n h_n: \{a_n\}_{n\in\NN^d_0}\in \ell_{0,\flat_{\sigma}}(\NN^d_0)\}
$$
were considered.
By
$\mathcal H_{\flat_{\sigma}, even}(\RR^d)$ and
$\mathcal H_{0,\flat_{\sigma}, even}(\RR^d)$
we denote the space of even
functions from $\mathcal H_{\flat_{\sigma}}(\RR^d)$ and  $\mathcal H_{0,\flat}(\RR^d)$, respectively.

\par

We  introduce  the space of  formal expansions
$$\mathcal G_{\flat_{\sigma}}(\RR^d)=\{\sum_{n\in\NN^d_0} a_n l_n: \{a_n\}_{n\in\NN^d_0}\in \ell_{\flat_{\sigma/2}}(\NN^d_0)\}
$$
and its Beurling counterpart
$$
\mathcal G_{0,\flat_{\sigma}}(\RR^d)=\{\sum_{n\in\NN^d_0} a_n l_n: \{a_n\}_{n\in\NN^d_0}\in \ell_{0, \flat_{\sigma/2}}(\NN^d_0)\}.
$$

Let us note that, for the arbitrary $\{a_n\}_{n\in\NN^d_0}\in \ell_{\flat_{\sigma/2}}(\NN^d_0)$, the expansion $\sum_{n\in\NN^d_0} a_n l_n$ is absolutely convergent, thus defining the function in $L^2(\mathbb R_{+}^d)$. The same holds for the Beurling case.



Now, Theorems \ref{even-spaces-1} and \ref{even-spaces-2} can be extended as follows.

\begin{thm} Let $\sigma>0$, and $v(x)=(x_1^2,\dots, x_d^2)$, $ x \in \RR^d$.
If $\ds f=\sum_{n\in\NN^d_0}{a_n(f)l_n}\in \mathcal{G}_{\flat_{\sigma}}(\RR^d_{+})$, then $\ds f\circ v\in \mathcal{H}_{\flat_{\sigma}, even}(\RR^d)$ and we have
$$f\circ v=\sum_{n\in\NN^d_0} b_{2n}h_{2n}$$
such that $\{b_{2n}\}_{n\in\NN^d_0}\in \ell_{\flat_{\sigma}}(\NN^d_0)$ and  (\ref{luh}) holds.
\end{thm}

\begin{proof}
Suppose that $\ds f =\sum_{n\in\NN^d_0}{a_n(f)l_n}\in \mathcal{G}_{\flat_{\sigma}}(\RR^d_{+})$, i.e.
$ \displaystyle |a_n(f)|\leq C \frac{r^{|n|}}{|n|!^{1/\sigma}}$,
$ n \in \NN^d_0,$ for some $C,r>0$.
Since   $\displaystyle \sum_{n\in\NN^d} |a_n|\leq C \sum_{n\in\NN^d} \frac{r^{|n|}}{|n|!^{1/\sigma}} < \infty$, the absolute convergence argument allows us to apply  formula (\ref{luh}). Then
by using   $ \sqrt{(2n)!}  \leq d^{|n|}n!$,
$ |k|! \cdot |n|! \leq |k+n|! $,
and
$$
\big|{k-{\bf 1/2}\choose k}\big| \leq 1,
$$
we obtain
$$
|b_{2n}(f)|\leq \big|\frac{(-1)^n  \pi^{\frac{d}{4}} \sqrt{(2n)!} }{2^{|n|}n!}\big|\cdot \sum_{k\in\NN_0^d} |a_{k+n}(f)|\cdot \big|{k-{\bf 1/2}\choose k}\big|
\leq C_1 \frac{r^{|n|}}{|n|!^{1/\sigma}}\sum_{k\in\NN_0^d}
\frac{r^{|k|}}{|k|!^{1/\sigma}}, \quad n\in\NN_0^d,
$$
for a suitable $C_1>0$.
Since $\displaystyle \sum_{k\in\NN_0^d}  \frac{r^{|k|}}{k!^{1/\sigma}} < \infty $, we get
$$|b_{2n}(f)|\leq C_2  \frac{r^{|n|}}{|n|!^{1/\sigma}}\leq C_3\frac{ \sqrt{r}^{2n}}{|2n|!^{1/(2\sigma)}},$$
for suitable constants
$C_2, C_3 > 0$, which gives
$$
\sup_{n \in \NN_0^d} |b_{2n}(f) | h^{2n}  |2n|!^{1/(2\sigma)}  < \infty,
$$
where $h:= \frac{1}{\sqrt{r}} >0$. Thus
 the sequence $\{b_{2n} \}_{n\in\NN^d}\in \ell_{\flat_{\sigma}}(\NN^d_0)$  gives rise to the Hermite expansion,
and defines a function in the space $\mathcal{H}_{\flat_{\sigma}, even}(\RR^d)$.
\end{proof}

\begin{thm} \label{flat-expansion}
Let $\sigma>0$,
$w(x)=(\sqrt{x_1},\dots \sqrt{x_d}),$ $x\in \RR_{+}^d$,
and let $\ds f=\sum_{n\in\NN_0^d} a_{2n}(f) h_{2n}\in \mathcal{H}_{\flat_{\sigma}, even}(\RR^d) $. Then $\ds f|_{\RR^d_{+}}\circ w $ belongs to $  \mathcal G_{\flat_{\sigma}}(\RR^d_{+})$,
$$f\circ v=\sum_{n\in\NN^d_0} b_{2n}h_{2n},$$
$\{b_n(f)\}_{n\in\NN_0^d}\in\ell_{\flat_{\sigma}}(\NN^d_0)$, and
 (\ref{hul}) holds.
\end{thm}

\begin{proof}
We sketch the main steps of the proof, and note that the absolute convergence can be justified by  calculations similar to the ones given in \cite{SSSB}.
Suppose that
$$\ds |a_{2n}(f)|\leq C \frac{r^{|2n|}}{|2n|!^{1/(2\sigma)}}, \quad
n\in\NN^d_0$$
for some $C,r>0$.
From \cite[(3.5)]{SSSB} (see also \eqref{hul})
  we get
$$|b_n(f)|\leq \Big|\frac{(-1)^n 2^{|n|}}{\pi^{d/4}}\Big|\cdot \sum_{k\in\NN_0^d}\Big|{k-{ \bf 3/2} \choose k}\Big|\cdot \Big|\frac{(-1)^{|k|} 2^{|k|}  (k+n)! a_{2k+2n} }{\sqrt{(2k+2n)!}}\Big|$$
$$\leq C_1  \frac{r^{2n}}{|2n|!^{1/(2\sigma)}}\sum_{k\in\NN^d}  \frac{r^{2k}}{|2k|!^{1/(2\sigma)}}\leq C_2 \frac{r^{2n}}{|n|!^{1/\sigma}}$$
for suitable  $C_i>0$, $i=1,2$. Thus
 the sequence $\{b_{2n} \}_{n\in\NN^d}\in \ell_{\flat_{\sigma}}(\NN^d_0)$  gives rise to the Laguerre expansion,
and defines a function $f|_{\RR_{+}^d}\circ w\in   \mathcal G_{\flat_{\sigma}}(\RR^d_{+})$, which concludes the proof.
\end{proof}

\par

Similar claims hold for $\mathcal G_{0,\flat_{\sigma}}(\RR^d)$,
we leave details for the reader.

\section{acknowledgement}
This research was supported by the Science Fund of the Republic of
Serbia, $\#$GRANT No. 2727, {\it Global and local analysis of operators and
distributions} - GOALS. The first author is supported by the Ministry of Science, Technological Development and Innovation of the Republic of Serbia Grant No. 451-03-65/2024-03/ 200169. The second author is supported by the project F10 financed by the Serbian Academy of Sciences and Arts. The third author is supported by the Ministry of Science, Technological Development and Innovation of the Republic of Serbia (Grants No. 451--03--66/2024--03/200125 $\&$ 451--03--65/2024--03/200125).


\begin{thebibliography}{10}




\bibitem{D0}
A.~ J.~ Duran, Laguerre expansions of Tempered Distributions and
Generalized Functions, J. Math. Anal. Appl. 150 (1990), 166-180.

\bibitem{DLP} A.~J.~Duran, A bound on the Laguerre polynomials, Studia Math. 100 (1991), 169-181.

\bibitem{D3}
A. ~J.~  Duran, The analytic functionals in the lower half plane as a
Gel'fand-Shilov space, Math. Nachr. 153 (1991), 145-167.

\bibitem{D4} A. ~J.~ Duran, Gelfand-Shilov spaces for the Hankel transform, Indag. Math. 3 (1992), 137--151.

\bibitem{D1}
A. ~J.~ Duran, Laguerre expansions of Gel'fand-Shilov spaces, J.
Approx. Theory 74 (1993), 280-300.



\bibitem{Ed}
A.~ Erdelyi, Higher Transcedentals Function, Vol. 2, McGraw-Hill,
New York, 1953.

\bibitem{GS2} I.~ M.~ Gel'fand, G.~ E.~ Shilov, Generalized Functions Volume 2 - Spaces of Fundamental and Generalized functions - Academic Press, 1968.



\bibitem{GPR0} T.~Gramchev, S.~Pilipovi\'{c}, L.~Rodino, \emph{Eigenfunction expansions in $\mathbb R^n$},  Proc. Amer. Math. Soc. \textbf{139} (2011), 4361--4368.





\bibitem{Sm} S.~ Jak\v{s}i\'{c}, B.~ Prangoski, Extension theorem of Whitney type for $S(\RR_+^d)$ by the use of the Kernel Theorem, Publ. Inst. Math. Beograd, 99(113)(2016), 59-65.

\bibitem{SSB} S.~ Jak\v{s}i\'{c}, S.~ Pilipovi\'{c}, B. ~Prangoski, G-type spaces of ultradistributions over $\RR_+^d$ and the Weyl pseudo-differential operators with radial symbols, RACSAM 111 (2017), 613-640.

\bibitem{SSBb} S.~ Jak\v{s}i\'{c}, S.~ Pilipovi\'{c}, B.~ Prangoski, Spaces of ultradistributions of Beurling type over $\RR_+^d$ through Laguerre expansions, Sarajevo J. Math, 12 (25) (2016), no. 2, suppl., 385-399.

\bibitem{SSSB} S.~ Jak\v{s}i\'{c}, S.~ Maksimovi\'{c}, S.~ Pilipovi\'{c}, B. Prangoski, Relations between Hermite and Laguerre expansions of ultradistributions over $\RR^d$ and $\RR^d_+$, J. Pseudo-Differ. Oper. Appl. 8 (2017), 275-296.

\bibitem{kothe1} G.~ K\"{o}the, Topological vector spaces I, Vol.I, Springer-Verlag, New York Inc., 1969.

\bibitem{kothe2} G.~ K\"{o}the, Topological vector spaces II, Vol.II, Springer-Verlag, New York Inc., 1979.

\bibitem{lan} M.~Langenbruch,
Hermite functions and weighted spaces of generalized functions,
Manuscripta Mathematica 119 (2006), 269-285.


\bibitem{LCP} Z.~Lozanov--Crvenkovi\'c, D.~Peri\v si\'c, M.~Taskovi\'c,
\newblock Hermite expansions of elements of Gelfand-Shilov spaces in quasianalytic and non quasianalytic case,
\newblock Novi Sad J. Math., 37 (2007), 129--147

\bibitem{MV}  R.~ Meise, D.~ Vogt, M.~S,~Ramanujan, Intoduction to Functional Analysis, Oxford University Press, London, 1997.



\bibitem{NR} F.~ Nicola, L.~ Rodino, Global Pseudo-Differential Calculus on Euclidean
Spaces, Pseudo-Differential Operators, 4, Birkh\" auser, Basel, 2010.


\bibitem{SP} S.~ Pilipovi\'c, Generalization of Zemanian spaces of
generalized functions which elements have series expansions,
SIAM J. Math. Anal. 17 (1986), 477-484.


\bibitem{SP1988} S. Pilipovi\'c, Tempered ultradistributions,
Bollettino della Unione Matematica Italiana, 7(2-B) (1988), 235--251.

\bibitem{SP1} S.~Pilipovi\'c, On the Laguerre expansions of generalized functions, C. R. Math. Rep. Acad. Sci. Canada 11 (1989),
 23-27.
 \bibitem{Pietsch} A.~ Pietsch, Nuclear locally convex spaces, Springer-Verlag, Berlin-Heildelberg- New York, 1972.




\bibitem{Thang} S.~Thangavelu,  Lectures on Hermite and Laguerre Expansions, Princeton University Press, Princeton New Jersey  (1993)

\bibitem{Toft1} J.~Toft, Images of function and distribution spaces under the Bargmann transform, J. Pseudo-Differ. Oper. Appl. 8 (2017) 83--139

\bibitem{TBM} J. Toft, D. G. Bhimani, R. Manna, Fractional Fourier transforms, harmonic oscillator propagators and Strichartz estimates on Pilipovi\'{c} and modulation spaces, Appl. Comput. Harmon. Anal. 67 (2023) 101580


\bibitem{Tr} F.~Tr\' eves, Topological Vector Spaces,
Distributions and Kernels, Dover Publications, New York, 1995.


\bibitem{VV} \DJ. Vu\v{c}kovi\'{c}, J. Vindas, Eigenfunction expansions of ultradifferentiable functions and ultradistributions in $\RR^d$, J. Pseudo-Differ. Opper. Appl. 7(4) (2016) 519- 531.


\bibitem{Wong} M.~ W. ~Wong, Weyl Transform, Springer-Verlag, New York, 1998.

\bibitem{Zayed} A.~ I. ~Zayed, Laguerre series as boundary values. SIAM J. Math. Anal. 13 (1982), 263-279.

\bibitem{Zemanian} A.~ H.~ Zemanian, Generalized integral transformations, Interscience, New York, 1968.
\end{thebibliography}
\end{document}